\documentclass[11pt]{article}
\usepackage{amsmath,amsthm,amsfonts,amssymb,mathrsfs,bm,graphicx,amscd,epsfig,psfrag,verbatim,hyperref}
\usepackage{enumerate, paralist, MnSymbol}
\usepackage[usenames]{color}
\usepackage{fancyhdr}
\parskip 	\bigskipamount

\newtheorem{theorem}{Theorem}[section]
\newtheorem{lemma}[theorem]{Lemma}

\newtheorem{proposition}[theorem]{Proposition}

\newtheorem{remark}{Remark}[section]
\newtheorem{example}{Example}[section]
\newtheorem{definition}{Definition}

\newtheorem{result}{Result}
\def\a{{\alpha}}
\def\N{\mathbb{N}}
\def\Z{\mathbb{Z}}

\def\R{\mathbb{R}}
\def\C{\mathbb{C}}
\def\E{\mathbb{E}}
\def\P{\mathbb{P}}

\def\del{\delta}

\def\F{\mathcal{F}}

\def\t{\theta}
\def\la{\lambda}

\def\D{\mathcal{D}}
\def\eps{\epsilon}

\def\ze{\mathcal{Z}}
\def\fa{f_{\alpha}}

\def\ph{\varphi}
\def\S{\mathcal{S}}

\def\out{\mathrm{out}}

\def\ol{\overline}
\def\var{\mathrm{Var}}

\def\X{\mathbb{X}}

\def\N{\mathbb{N}}

\def\wt{\widetilde}

\def\F{\mathcal{Z}}

\def\out{\mathrm{out}}
\def\inn{\mathrm{in}}

\def\wt{\widetilde}

\def\la{\lambda}
\def\a{\alpha}
\def\ol{\overline}

\def\eps{\varepsilon}

\def \del{\delta}

\def \Z{\boldsymbol{Z}}

\def \X{\boldsymbol{X}}

\def\E{\mathbb{E}}

\def\ja{(j!)^{\alpha}}
\def\jh{(j!)^{\alpha/2}}

\def\L{\Delta}

\def\half{\frac{1}{2}}

\def\X{\mathcal X}
\def\Var{\mathrm{Var}}

\renewcommand{\l}[0]{\left }
\renewcommand{\r}[0]{\right}

\def\suchthat{\! : \! }
\def\Rig{\mathcal R}
\def\mb{\mbox}
\renewcommand\Z{\mathbb Z}
\def\uu{\mathbf u}

\usepackage{palatino}

\newcommand\numberthis{\addtocounter{equation}{1}\tag{\theequation}}

\makeatletter
\renewcommand*{\@cite@ofmt}{\hbox}
\makeatother

\begin{document}
\title{Rigidity hierarchy in random point fields: random polynomials and determinantal processes}
\author{
\begin{tabular}{c}
Subhroshekhar Ghosh \\ National University of Singapore \\ subhrowork@gmail.com
\end{tabular}
\and
\begin{tabular}{c}
Manjunath Krishnapur \\ Indian Institute of Science \\ manju@iisc.ac.in
\end{tabular}
}
\date{}
\maketitle
\begin{abstract} In certain point processes, the configuration of points outside a bounded domain determines, with probability 1, certain statistical features of the points within the domain. This notion, called rigidity, was introduced in \cite{GP}. In this paper, rigidity and the related notion of tolerance are examined systematically and  point processes with rigidity of various degrees are introduced. Natural classes of point processes such as determinantal point processes,  zero sets of Gaussian entire functions and perturbed lattices are examined from the point of view of rigidity, and general conditions are provided for them to exhibit specified nature of spatially rigid behaviour. In particular, we examine the rigidity of determinantal point processes in terms of their kernel, and demonstrate that a necessary condition for determinantal processes to exhibit rigidity is that their kernel must be a projection. We introduce a one parameter family of point processes which exhibit arbitrarily high levels of rigidity (depending on the choice of parameter value), answering a natural question on point processes with higher levels of rigidity (beyond the known examples of rigidity of local mass and center of mass). Our one parameter family is also related to a natural extension of the standard planar Gaussian analytic function process and their zero sets.
\end{abstract}

\section{Introduction}
 \label{introduction}

\subsection{Background} \label{sec:background}
A simple point process on $\R^{d}$ is a random locally finite subset of $\R^d$. To be precise, let $\mathcal{S}$ denote the space of locally finite point sets in $\R^d$. By identifying a set with its counting measure, $\mathcal S$ becomes an open subset of $\mathcal M_p$, the space of all integer-valued locally finite Borel measures on $\R^d$. Equipped with the vague topology, $\mathcal M_p$ turns out to be a Polish space. A random variable taking values in $\mathcal M_p$ (with probability $1$)  is called a point process, and a random measure that  takes values in $\mathcal S$ (with probability $1$) is called a simple point process. A slightly more general definition is given in Section~\ref{sec:setup}.  For more detailed discussion, we refer the interested reader to Chapter~1 of \cite{bbkbook} or the extensive works \cite{DV}, \cite{Ka}.  
 
The most well-studied models among point processes on Euclidean spaces are the Poisson point processes. In a Poisson process, for any Borel subset $D\subseteq \R^{d}$, the configuration of points of the point process that fall inside $D$ and the configuration of points that fall inside $D^{c}$ are independent. For a general point process such independence is no longer valid, and in fact, the study of spatial conditioning in such processes becomes a an important and highly non-trivial question. An extensive literature on the topic has developed over the years, including on tractable generalisations of independence such as Gibbs-type properties \cite{Ge}. 

In the work \cite{GP}, a surprising phenomenon that is in a sense diametrically opposite to independence was shown to occur in certain point processes -- namely, some features of the configuration of points inside $D$ (eg., the number of points that fall in $D$) may be determined, almost surely, by the configuration of points in $D^{c}$. This phenomenon was referred to as {\em rigidity} in \cite{GP}. In other words, some statistic $X$ that is a measurable function of the points in $D$ may also be shown to be a measurable function of the points in $D^c$ (in which case $X$ is said to be rigid). The precise definitions of rigidity and related concepts will be given later in the article; for the present introductory discussions we will content ourselves with the idea as explained above.

 It was shown in \cite{GP} that for the infinite Ginibre ensemble (recalled more precisely in Result~\ref{res:ginibre}) on $\R^{2}$, the number of points inside a bounded region $D$ is determined by the configuration of points outside $D$ (i.e., by the configuration of points in $D^c$). It was also shown that for the zeroes of the canonical Gaussian entire function (see Result~\ref{res:GEF}), the number of points as well as the center of mass of the points in a bounded region are determined by the configuration of points outside. In both cases it was shown, in a precise sense, that nothing more than this can be determined by the configuration outside.

The goal of this paper is to systematically explore the extent of rigidity in several natural point processes. In Section \ref{sec:setup}, we give definitions of rigidity and tolerance following \cite{GP}, but in a more general setting. In Section \ref{sec:models}, we recall the known examples and related results in the literature. Further, we mention certain applications to the study of point processes to which these notions have been crucial. In Section~\ref{sec:ourresults} we state the new results of this paper. In subsequent sections we give proofs of these results.

\subsection{ Setup and fundamental notions}  \label{sec:setup}
In this section, we give definitions of rigidity and tolerance following \cite{GP}, albeit in a more general setting. 

The definition of point process on $\R^d$ recalled in the introduction generalizes in a straightforward way to point processes on any locally compact, second countable Hausdorff space (or l.c.s.h space in short).  Let  $\Xi$ be a l.c.s.h. space. Let $\mathcal M$ denote the space of Radon measures on $\Xi$, a Polish space with the topology of vague convergence. Let $\mathcal M_{p}\subseteq \mathcal M$ denote the closed subset  of integer-valued Radon measures on $\Xi$, and let $\S\subseteq \mathcal M_{p}$ denote the relatively open subset  of counting measures of locally finite subsets of $\Xi$.  As Borel subsets of $\mathcal M$, both $\mathcal M_p$ and $\mathcal S$ inherit Borel sigma-algebras. A random variable $\Pi$ defined on some probability space $(\Omega,\mathcal F,\P)$ is called a {\em random measure} if it takes values  in $\mathcal M$, a {\em point process} if it takes values (with probability 1) in $\mathcal M_p$ and a {\em simple point process} if it takes values (w.p. 1) in $\S$.  For more details on point processes, we refer the reader to \cite{bbkbook} (Chapter~1), \cite{DV}, \cite{Ka}.

If $D\subseteq \Xi$ is a Borel set, let $\S_{D}=\{\theta \in \S \suchthat \theta(D^{c})=0\}$. For a point process $\Pi$ on $\Xi$, let $\Pi_{D}$ denote the restriction of $\Pi$ to $D$, i.e., the point process on $\Xi$ satisfying $\Pi_{D}(A)=\Pi(A\cap D)$.  Let $\sigma\{\Pi_{D}\}$ denote the sigma-algebra generated by $\Pi_{D}$. Random variables measurable with respect to $\sigma\{\Pi_{D}\}$ are precisely those of the form $f(\Pi_{D})$ where $f:\S_{D}\mapsto \R$ is  measurable. 

The key  object in our study is the sigma-algebra $\Rig_{D}:=\overline{\sigma(\Pi_{D})}\cap \overline{\sigma(\Pi_{D^{c}})}$ where $\overline{\mathcal F}$ denotes the completion of a sigma-algebra $\mathcal F$. If $Y$ is an $\Rig_{D}$-measurable random variable, then there exist measurable functions $f_{\inn}:\S_{D}\mapsto \R$ and
$f_{\out}:\S_{D^{c}}\mapsto \R$ such that $Y=f_{\inn}(\Pi_{D})$ $a.s.$ and $Y=f_{\out}(\Pi_{D^{c}})$ $a.s.$ Thus, $\Rig_{D}$-measurable random variables are precisely those features of $\Pi_{D}$  which can be inferred exactly (with probability one), from a knowledge of the configuration $\Pi_{D^{c}}$.
\begin{definition}[Rigidity]
A measurable function $f:\S_{D}\mapsto \R$ is said to be rigid for the pair $(\Pi,D)$ if $f(\Pi_{D})$ is $\Rig_{D}$-measurable.
\end{definition}
Clearly, this is interesting only if $f(\Pi_{D})$ is not a constant. In that case,  $\Rig_{D}$ contains events with non-trivial probability.

 The most basic feature of interest is the number of points in $D$, i.e., $N_{D}:\S_{D}\to \R$ defined as $N_{D}(\Pi)=\Pi(D)$. If $N_{D}$ is rigid for all pre-compact $D\subseteq \Xi$, then we say that $\Pi$ exhibits {\em rigidity of numbers}.

Once we establish that there are rigid functions, the natural question is, how many? This is captured by the following definition. We write  $\sigma\{Y_{i}\suchthat i\in I\}$ for the sigma-algebra generated by  $Y_{i}$, $i\in I$.
\begin{definition}[Tolerance]\label{def:tolerance} Let $f_{1},\ldots ,f_{m}:\S_{D}\mapsto \R$ be measurable functions. We say that the pair $(\Pi,D)$ is tolerant subject to $\{f_{1},\ldots ,f_{m}\}$ if $\sigma\{f_{1}(\Pi_{D}),\ldots ,f_{m}(\Pi_{D})\}$ and $\Rig_{D}$ are equal up to $\P$-null sets.
\end{definition}
The definition entails that each $f_{i}$ is rigid for $(\Pi,D)$ and for every $A\in \Rig_{D}$, there is a $B\in \sigma\{f_{1}(\Pi_{D}),\ldots ,f_{m}(\Pi_{D})\}$ such that $\P(A\Delta B)=0$. Of course, we may allow $m=\infty$ and make the same definition. This definition captures the idea that nothing more than $f_{i}(\Pi_{D})$, $i\le m$, can be determined by the outside configuration $\Pi_{D^{c}}$. Or more precisely, any feature of the point process inside $D$ that can be determined by the outside configuration is equal (almost surely) to a function of these $m$ features.

However, there are other ways to capture the same idea. In most of the cases that we encounter, the following stronger form of tolerance holds.
\begin{definition}[Tolerance - strong form] \label{def:tolerancestrong2}
In the language of Definition~\ref{def:tolerance}, we say that $(\Pi,D)$ is strongly tolerant  subject to $\{f_{1},f_{2},\ldots ,f_{m}\}$ if, almost surely, the conditional distribution of $\Pi_{D}$ given $\Pi_{D^{c}}$ is mutually absolutely continuous to the conditional distribution of $\Pi_{D}$ given $\sigma\{f_{1}(\Pi_{D}),\ldots ,f_{m}(\Pi_{D})\}$.
\end{definition}
There are other possible definitions of tolerance that lie between these two. In section~\ref{sec:thiswillberewritten} we discuss these and also explain the relationship between tolerance and strong tolerance. In particular, it is shown that strong tolerance implies tolerance, as the name suggests.

In general, the extent of rigidity could depend not only on the underlying point process but also on the sub-domain $D$. In almost all cases considered in this paper, the point process is on $\R^{2}$ and the extent of rigidity is the same for all bounded subsets $D$.   An exception is in non-projection determinantal processes of Theorem~\ref{detrig}, see Remark~\ref{rem:inclusion}.

\subsection{Existing literature on rigidity and tolerance}  \label{sec:models}
In this section, we discuss significant strands of results available in the literature related to the topic of rigidity and tolerance, and applications thereof.
\subsubsection{Known models of rigidity and tolerance}
We begin with surveying available results on rigidity and tolerance phenomena for standard point processes.
A Poisson process has no rigidity. In other words, $\Rig_{D}$ is trivial for any $D$. It is interesting and non-trivial that there are point processes with any rigidity at all. We quote three sample results. The first two are from \cite{GP}  while the last is from \cite{G-1}.

{\it  Infinite Ginibre ensemble in the complex plane}: This is the determinantal point process defined by the kernel $e^{z\bar{w}}$ with respect to the measure $\frac{1}{\pi}e^{-\frac12 |z|^{2}}$ on the complex plane. For the definition of a determinantal point process, we refer the reader to \cite{soshnikov} or chapter 4 of \cite{HKPV} (we give a brief description in Section~\ref{sec:ourresults}). The infinite Ginibre ensemble is a translation and rotation invariant simple point process on the plane.
\begin{result}[\cite{GP}] \label{res:ginibre}The infinite Ginibre ensemble exhibits rigidity of numbers. Furthermore, the infinite Ginibre ensemble is strongly tolerant subject to the number of points for any bounded domain $D \subset \R^2$.\emph{}
\end{result}
 Informally, the result says that for any bounded set $D$, from the  knowledge of $\Pi_{D^{c}}$, it is possible to determine the number of points in $\Pi_{D}$, but nothing more can be determined.

{\it Zero set of the canonical Gaussian entire function}: Let $\Pi$ be the zero set of the  random entire function $\sum_{n=0}^{\infty}\xi_{n}\frac{z^{n}}{\sqrt{n!}}$ where $\xi_{n}$ are i.i.d. standard complex Gaussian random variables (each $\xi_{n}$ has density $\frac1\pi e^{-|z|^{2}}$). Then, $\Pi$ is a translation and rotation invariant simple point process on the plane (see \cite{ST1} for details); for more on translation invariant zero sets of Gaussian analytic functions, see e.g. \cite{Fe}.

\begin{result}[\cite{GP}] \label{res:GEF} The zero set of the Gaussian entire function exhibits rigidity of the number and the center of mass of the points in any bounded domain $D \subset \R^2$.  Furthermore, for any bounded domain $D \subset \R^2$, this zero process is strongly tolerant subject to the number and the center of mass of the points.\emph{}
\end{result}
Informally, this means that for any bounded set $D$, from the knowledge of $\Pi_{D^{c}}$, it is possible to determine the number of points in $\Pi_{D}$ and the center of mass of these points, but nothing more.

{\it Sine-kernel determinantal point process on $\R$}: This is the determinantal point process with kernel $\frac{\sin(x-y)}{x-y}$ with respect to the Lebesgue measure on $\R$. This is a translation invariant simple point process (see section~2 of \cite{soshnikov}) that arises in the study of Hermitian random matrices.

\begin{result}(\cite{G-1}) \label{res:sinekernel} The sine-kernel determinantal point process exhibits rigidity of numbers.
\end{result}
A similar result was shown for a wide class of translation  invariant determinantal point processes on the real line. These determinantal processes (see Soshnikov~\cite{soshnikov} and Lyons and Steif~\cite{LySt}) have kernels of the form  $\hat{g}(x-y)$ where $g$ is the indicator of a subset of $\R^d$ having finite positive Lebesgue measure, see \cite{G-1} for details. 
Number rigidity for sine beta processes was proved in \cite{DHLM}.

Bufetov~\cite{buf-3} demonstrated rigidity of numbers for determinantal point processes  with Airy, Bessel and Gamma kernels. These point processes are not translation invariant. In addition, Bufetov and Qiu~\cite{buf-1}, have shown rigidity for a class of determinantal processes in the plane. Their result has an overlap with, but does not subsume, Theorems~\ref{compact} and \ref{dpprt} that we prove in this paper. In addition, it is worth mentioning that in \cite{buf-2}, Bufetov, Dabrowski and Qiu study one-sided rigidity. In our terminology, this corresponds to taking $\Xi=\R$ and $D=[a,\infty)$. This result may be the only known result that proves rigidity for unbounded sets, although one may mention that analogous results (theorems of Szeg\"{o} and Kolmogorov on the complete predictability of a Gaussian process from its past) are known in the theory of Gaussian processes, see~ \cite{DymMackean}. Rigidity phenomena have subsequently been considered widely general set-ups, mostly in the case of determinantal processes, see e.g. Qiu (\cite{Qi})  and Hiraoka, Shirai and Trinh (\cite{HST}).

\subsubsection{Related literature} In \cite{OsSh}, Osada and Shirai showed that for the Ginibre ensemble, the Palm measures with respect to different point sets are mutually absolutely continuous if the conditioning set of points have the same cardinality, and are mutually absolutely continuous otherwise. Such dichotomy is similar in spirit to, and is, in fact, closely related with the rigidity phenomena under our consideration. There has been recent interest in examining the extent to which rigidity can occur in stochastic systems, including applications to spectrally constrained random systems, \textit{stealthy} models, and maximal rigidity; see e.g. the recent works \cite{GhL, AGL} and also the paper \cite{KiNi} that are pursuant to an earlier arxiv version of our present paper \cite{GhK}. For connections between the strength of rigidity in a point process and its \textit{Palm measures}, we refer the reader to \cite{G-2}. In the more general area of spatial conditioning in point processes, important progress has been obtained recently in \cite{BQS} in the setting of determinantal processes, which also establishes the Lyons-Peres completeness conjecture. 

\subsubsection{Applications to other aspects of point processes:} The study of rigidity, although natural, is not an end in itself. Rigidity phenomena have  been exploited to answer questions on point processes and stochastic geometry that are interesting in their own right.

In \cite{G-1}, rigidity of the sine kernel process (Result~\ref{res:sinekernel}) was exploited to answer a completeness question related to exponential functions coming from the point process. In \cite{GKP}, the authors used an understanding of the rigidity behaviour of the Ginibre ensemble and the zero set of the Gaussian entire function   (Result~\ref{res:ginibre} and Result~\ref{res:GEF}) in order to study continuum percolation on these models. In particular, they showed the  existence of a non-trivial critical radius for percolation, and established the uniqueness of the infinite cluster in the supercritical regime.

In \cite{Os}, an understanding of  quasi-Gibbs property, which has a somewhat similar  flavour to rigidity and tolerance, was used to define an infinite particle SDE for invariant dynamics on the Ginibre process. To execute a similar programme for invariant dynamics on the zero set of the Gaussian entire function, one faces new challenges involving the higher level of rigidity in Result~\ref{res:GEF}, and it is a topic of current research.

\section{Our results}\label{sec:ourresults}
Now we present our results along with some remarks on our motivations and on what more could be true but are not known.

\subsection{Point processes with higher levels of rigidity and $\a$-GAFs} \label{sec:higher_rig}
The Results~\ref{res:ginibre} and \ref{res:GEF} demonstrate point process that are rigid only with respect to the local mass or local mass and center of mass respectively. The primary  question that we explore in this section is whether there are there natural point processes that exhibit higher levels of rigidity.  This could, for example, entail that for some $k\ge 3$ and any bounded set $D$, the $k$-th moment (i.e., the sum of $k$th powers) of the points of the process that fall in $D$ can be determined by the outside configuration $\Pi_{D^{c}}$.

We answer this natural  question in the affirmative, by constructing zero sets of Gaussian entire functions with arbitrarily high degrees of rigidity. To this end, we first define $\a$-Gaussian analytic functions (\textit{abbrv.} $\a$-GAF).

\begin{definition}
  For a real number $\a>0$,  the $\a$-GAF is the random entire function  \[f_\a(z):=\sum_{k=0}^\infty \frac{\xi_k}{(k!)^{\a/2}}z^k\]
  where $\xi_{k}$ are i.i.d. standard complex Gaussian random variables.
 \end{definition}
For $\a=1$, we are reduced to the canonical planar Gaussian entire function whose zero set is a translation-invariant point process (see Proposition~2.3.4. in \cite{HKPV}). For any $\a\not=1$, the zero set of $\a$-GAFs is not translation-invariant in the plane. Figure~\ref{fig:alphagafzerosets} demonstrates how the parameter $\alpha$ affects the distribution of the zeroes.


For stating our main theorem of this section, we recall that the (holomorphic) $p$-th moment (or the moment of order $p$, for any integer $p \ge 0$) of a (finite) point set $\{z_1,\cdots,z_m\} \subset \R^2=\C$ is given by $z_1^p + \cdots +z_m^p$, where the $z_i$ are treated as complex numbers. By the first $p$ moments of such a point set, we imply the set of its moments of order $\{0,1,\ldots,p-1\}$.

 \begin{theorem}
  \label{hierig}
  The zero set of the $\a$-GAF exhibits rigidity of the first $k$ moments of its points in any bounded domain $D \subset \R^2$ whenever $\a \in (\frac{1}{k} , \frac{1}{k-1}]$.
 \end{theorem}

In the forthcoming paper \cite{G-3}, we investigate and answer in the affirmative the question of the strong tolerance of the $\a$-GAF zeros given the first $k$ moments of its points in a bounded domain (where $\a \in (\frac{1}{k} , \frac{1}{k-1}]$, as in the statement of Theorem \ref{hierig}).

\subsection{Rigidity of determinantal point processes} \label{sec:DPP}
Result~\ref{res:ginibre} and Result~\ref{res:sinekernel} exhibit determinantal point processes on $\R^{2}$ and on $\R$ that exhibit rigidity of numbers, and strongly tolerant given the number of points (i.e., for any bounded set $D$, the number of points of the process that fall inside $D$  is determined by the outside configuration $\Pi_{D^{c}}$, but nothing more). The natural question that arises pertains to the rigidity and tolerance properties exhibited a general determinantal point process. To articulate our result, we first recall the precise notion of a determinantal point process.

A determinantal point process (henceforth abbreviated as d.p.p.) on $\Xi$ with a Hermitian kernel $K:\Xi\times \Xi\to \C$ (Hermitian means that $K(x,y)=\overline{K(y,x)}$) and a (Radon) measure $\mu$, is a point process whose joint intensities (also called correlation functions)  with respect to $\mu$ are given by
\begin{align*}
\rho_{k}(x_{1},\ldots ,x_{k}) = \det\l(K(x_{i},x_{j})\r)_{i,j\le k}
\end{align*}
for any $k\ge 1$ and any $x_{1},\ldots ,x_{k}\in \Xi$. We refer to \cite{HKPV}, chapters 1 and 4, for more details of these notions.  All information about a determinantal process is in the kernel $K$, once the  measure $\mu$ is fixed. For $D\subseteq \Xi$, we define the integral operator $\mathcal K_{D}:L^{2}(D,\mu)\mapsto L^{2}(D,\mu)$ by
\[
(\mathcal K_{D} f)(x)=\int_{D}K(x,y)f(y)d\mu(y).
\]
We say that the kernel $K$ is locally of trace-class if for every $D\subseteq \C$ that is pre-compact, the operator $\mathcal K_{D}$ is of trace-class. We write $\mathcal K$ for $\mathcal K_{\Xi}$.

It is well known that the kernel and background measure completely determine the distribution of a determinantal point process (Lemma~4.2.6 in \cite{HKPV}). Hence in principle, the rigidity properties of a determinantal process are encoded in its kernel and background measure, possibly in a highly indirect and complicated manner. 

A natural problem, therefore, is to understand the rigidity and tolerance behaviour of a determinatal point process in the form of a simple criterion expressed in terms of the kernel.  We provide a partial answer to this question by specifying conditions on the pair $(K,\mu)$ that mandate a certain rigidity behaviour of the determinantal process. More specifically, we investigate conditions under which a determinantal process has rigidity of numbers. Recall that this means that for any $D\subseteq \Xi$ whose closure is compact, the counting function $N_{D}:\S_{D}\mapsto \R$, is rigid.

To state our  result, we recall the following result of Machchi and Soshnikov (see Theorem~3 in Soshnikov~\cite{soshnikov}): Let $K$ be locally of trace-class and Hermitian. Then, a determinantal point process with the pair $(K,\mu)$ exists if and only if $0\le \mathcal K\le I$. Equivalently, the spectrum of $\mathcal K$ is contained in $[0,1]$. Clearly $\mathcal K$ defines a projection operator if and only if its spectrum is contained in $\{0,1\}$.

Now we are ready to state a necessary condition for a determinantal point process to exhibit rigidity of numbers.
 \begin{theorem}
  \label{detrig}
  Let $\Pi$ be a determinantal point process with kernel $K$ and background measure $\mu$. Then $\Pi$ exhibits rigidity of numbers only if the integral operator $\mathcal K$ associated to $K$ is a projection operator.
  \end{theorem}

\begin{remark} \label{rem:inclusion} 
The statement of Theorem \ref{detrig} primarily entails the contrapositive of being number rigid, in the sense that, if $\mathcal K$ is not a projection operator, then there exists a  precompact Borel set $D$ (possibly depending on $\mathcal K$) such that the point count in $D$ (i.e., $N_D(\Pi)$) is not rigid. However, it may be easily  seen that, if the point count in a domain $D_1$ (i.e., $N_{D_1}(\Pi)$)  is not rigid, then the same will hold true for the point count $N_{D_2}(\Pi)$ for any domain $D_2$ such that $D_1 \subseteq D_2$ (for a detailed argument, see Lemma \ref{lem:inclusion} in Section \ref{sec:thiswillberewritten}). 

However, it is not true in general that  rigidity of numbers holds for all bounded sets.  For example, let  $D_0$ be a bounded open set in $\R^d$, and suppose that $\mathcal K_{D_0}$ is a projection  of finite rank $n$ and that $\mathcal K_{D_0^c}$ is not a projection. Examples are easy to construct, for example by taking  finite rank projection operators $\mathcal K_1$ and $\mathcal K_2$ on $L^2(D_0)$ and $L^2(\D_0^c)$ respectively, and setting $\mathcal K=\mathcal K_1\oplus \frac12 \mathcal K_2$ (the direct sum; in terms of the associated kernels, $K=K_1$ one $D_0$ and $K=\frac12 K_2$ on $D_0^c$). Then $\mathcal K$ is a not a projection, and the corresponding DPP does fail to have number rigidity for some $D$. However, $\Pi$ has exactly $n$ points in $D_0$ and hence for any $D\subseteq D_0$, we do have number rigidity, since $N_D(\Pi)=n-N_{D^c\cap D_0}(\Pi)$, which is measurable w.r.t. $\sigma(\Pi_{D^c})$. 
\end{remark}

It is an interesting question as to whether, conversely, every projection determinantal point process exhibits rigidity of numbers. This is certainly true for  the infinite Ginibre ensemble (Result~\ref{res:ginibre}) and for the sine-kernel process on the line (Result~\ref{res:sinekernel}), but is not true in complete generality.

A counter-example is provided by the  Bergman determinantal point process, which is a determinantal process on the unit disk $\mathbb D=\{z\in \mathbb C \! : \! |z|<1\}$ whose kernel is the Bergman kernel $K(z,w)=(1-z\bar{w})^{-2}$ and  $d\mu(z)=\pi^{-1}dm(z)$ is the normalized Lebesgue measure on the unit disk. The corresponding operator is the projection from $L^{2}(\mu)$ onto the subspace of square integrable holomorphic functions. Holroyd and Soo~\cite{HS} have shown that this determinantal process is  insertion tolerant and deletion tolerant\footnote{A point process $\Pi$ on $\R^{d}$ is said to be insertion tolerant  if for any set subset  $A\subseteq\R^{d}$ with positive, finite Lebesgue measure, if a point $U$ is sampled uniformly from $A$ and added to $\Pi$, the distribution of the resulting point process $\Pi+\del_{U}$, is absolutely continuous to the distribution of $\Pi$. Poisson process is an obvious example. Evidently, an insertion tolerant point process does not have rigidity of numbers.}. Consequently it does not display rigidity of numbers.

However, this is an example where the ambient space (the hyperbolic plane) exhibits non-Euclidean geometry (in particular, it displays hyperbolic or Lobachevskian geometry \cite{Terras,And,Lob}). This still leaves open the question of whether a translation-invariant determinantal point process on a $\R^{d}$ given by a projection kernel  necessarily exhibits rigidity of numbers, or more generally, what can be said about the rigidity  properties of such point processes. We do not have a complete answer to this question, but in this work we are able to prove rigidity of numbers (or the lack of it) for a significant class of determinantal point processes in the plane. This is taken up for consideration in the next section.

\subsection{Radial determinantal processes on the plane} \label{sec:radial}
 Let $\gamma$ be a radially symmetric probability measure on the complex plane such that $c_{j}^{-1}:=\int_0^\infty |z|^{2j} d\gamma(z)<\infty$ for all $j$. Then $\{\sqrt{c_{j}}z^{j} \! : \! j\ge 0\}$ is an orthonormal set in $L^{2}(\gamma)$. It is an orthonormal basis for the subspace of holomorphic functions in $L^{2}(\gamma)$ and the projection operator onto this closed subspace is given by the kernel $K(z,w)=\sum_{j=0}^{\infty} c_j z^j\ol{w}^j$.  By the result of Machchi and Soshnikov quoted earlier, there is a unique  determinantal process with kernel $K$ and measure $\gamma$. These processes include the infinite Ginibre ensemble (when $\gamma$ is standard complex Gaussian measure on the plane) and the Bergman determinantal process (when $\gamma$ is normalized uniform measure on the unit disk).

 Before stating the main result in this section, we introduce the following quantities.
 \begin{align*}
 \mu_j:=\frac{c_j}{c_{j+1}}, &\qquad \sigma_j:=\frac{\mu_{j+1}}{\mu_j}-1, \qquad
 \nu_j := \; \frac{\mu_{j+3}\mu_{j+2}\mu_{j+1}}{\mu_j^3} - 4 \frac{\mu_{j+2} \mu_{j+1}}{\mu_j^2} + 6 \frac{\mu_{j+1}}{\mu_j} -3.
 \end{align*}
 In the setting described so far, we prove the following theorems.
 \begin{theorem}
 \label{compact} Let $\Pi$ be the determinantal point process  on $\C$ with kernel $K$ with respect to the radial background measure $\gamma$. If $\sum_{j=1}^{\infty}\sigma_j < \infty$, then the process $\Pi$ does not exhibit rigidity of numbers.
\end{theorem}
 \begin{theorem}
 \label{dpprt}
 Let $\Pi$ be the determinantal point process on $\C$ with kernel $K$ with respect to the radial background measure $\gamma$. Assume that \begin{inparaenum}[(1)] \item $\mu_j \to \infty$ as $j \to \infty$, \item $\sum_{j=1}^{\infty}\nu_j<\infty$ and \item$\sum_{j=1}^{k} \mu_j^4 \nu_j= o(\mu_k^4)$ as $k \to \infty$.  \end{inparaenum}  Then $\Pi$ exhibits rigidity of numbers.
\\
\end{theorem}
The condition $\sum_{j}\sigma_{j}<\infty$ in Theorem~\ref{compact} is satisfied if $\mu_j=A+Bj^{-1}+ O(j^{-2})$. The conditions in Theorem~\ref{dpprt} are satisfied if $\mu_j=Cj^a(1+A j^{-1} + O(j^{-2}))$. The latter is certainly the case if $\mu$ has density proportional to $|z|^{\alpha}e^{-|z|^{\beta}}$. This includes the infinite Ginibre ensemble with $\alpha=0$ and $\beta=2$.

We believe that the additional moment assumption $\sum_{j=1}^{\infty}\mu_j^{-a}<\infty$ for some integer $a \ge 1$ would suffice to establish strong tolerance for the process $\Pi$ in Theorem \ref{dpprt} for any bounded domain $D$ subject to its particle number. However, the rigorous demonstration of such a result would take us too far afield with regard to techniques and methods, and it would be appropriate to leave such considerations for a different occasion.




While Theorems \ref{compact} and \ref{dpprt} cover a good class of radially symmetric determinantal processes, it would be definitely very interesting to sharpen the dichotomy present in these two theorems. In particular, it be of great interest to extend this dichotomy and establish that it is the compactness or the non-compactness of the support of the measure $\mu$ that determines the rigidity behaviour of the determinantal process (e.g. as in Theorems \ref{compact} and \ref{dpprt}).

\subsection{Perturbed lattice models} \label{sec:perturbed_lattice}
 A perturbed lattice  is the point process $\sum_{k\in \Z^d}\delta_{k+\xi_{k}}$, whose points are $\{k+\xi_{k}\! : \! k\in \Z^{d}\}$ where $\xi_{k}$ are independent $\R^{d}$-valued random vectors. Thus, heuristically speaking, by a perturbed lattice we mean a point process obtained by perturbing the points of a lattice by independent random variables. 

In one dimension, if the perturbations are i.i.d., non-degenerate and have finite variance, then it can be shown that the resulting point process exhibits rigidity of numbers, and is strongly tolerant given the particle number for any bounded domain.  For i.i.d. lattice perturbations of $\Z^2$, the situation is similar these results are due to Holroyd and Soo \cite{HS}. Remarkably, for i.i.d.  perturbations by isotropic Gaussians in dimensions $d \ge 3$, it is known (\cite{PS}) that there is a phase transition in the variance of the perturbing Gaussian. Then, there is a critical $\sigma_c$ such that if the variance of the perturbing Gaussians is more than $\sigma_c^2$, then the point process is deletion tolerant\footnote{A point process $\Pi$ is said to be deletion tolerant if for any randomly chosen point $Z$ of $\Pi$, the distribution  of the point process $\Pi-\del_{Z}$ got by deleting the point at $Z$, is absolutely continuous to the distribution of $\Pi$.}, while for  perturbations with variance less than $\sigma_c^2$ it exhibits rigidity of numbers.

In this work we investigate the setting of dimension $1$ and Gaussian perturbations, but where the perturbations are no longer identically distributed.  The following theorem entails that if we posit a power law growth for the variances of the perturbations, the model undergoes a phase transition in the rigidity behaviour as follows.

\begin{theorem}
 \label{power}
Consider the point process $\X=\sum_{k\in \Z}\delta_{k+Z_k}$ on $\R$, where $Z_k\sim N(0,\sigma_\beta(k))$ are independent and $\sigma_\beta(k) = |k|^\beta$. Then,
  \begin{enumerate}
  \item For $\beta>1/2$, $\X$ does not exhibit rigidity of numbers. 
  \item For $\beta \le 1/2$, $\X$ exhibits rigidity of the number of points in any bounded domain  and strong tolerance given this  number.
  \end{enumerate}
\end{theorem}

 \section{Rigidity of $\a$-GAFs: Proof of Theorem~\ref{hierig}}
 Let $\fa$ be the $\a$-Gaussian analytic function (abbreviated as $\a$-GAF) defined by \[ \fa(z)= \sum_{j=0}^{\infty} \frac{\xi_j}{\jh} z^j. \] In this section, $\ze_{\a}$ will denote the point process of zeroes of $\fa$. Our goal  is to establish that  $\ze_\a$ exhibits rigidity of the first $k_{\a}$ moments in any bounded domain where $\frac{1}{k_{\a}}<\a \le \frac{1}{k_{\a}-1}$. We formulate this assertion in the form of the following Lemma.
 
 \begin{lemma}
\label{rigidity}
 Let $D\subset \C$ be a bounded open set, and let $\mathcal{S}_{D^c}$ denote the space of locally finite point configurations on $D^c$. For each $0 \le k \le \lfloor 1/\a \rfloor$,  there is a map $S_k : \mathcal{S}_{D^c} \to \C$ such that the moment $\sum_{z \in \ze_\a \cap D} z^k = S_k(\ze_\a \cap D^c)$ a.s.
 \end{lemma}

Observe that Lemma \ref{rigidity} will imply Theorem \ref{hierig}, so henceforth we focus on proving Lemma \ref{rigidity}.


As in nearly all proofs of rigidity, we would be interested in proving upper bounds on the variance of appropriate linear statistics of the point process of zeroes. We formulate this in the form of a lemma on the variance growth of linear statistics.

\begin{lemma} \label{lem:varlinstat}
Let $\ph$ be a $C_c^{\infty}$ function on $\C$ such that \begin{inparaenum}[(a)] \item $|\Delta \ph|$ is radial and \item $\Delta \ph$ vanishes in a neighbourhood of the origin. \end{inparaenum}
Let $L>0$, and define $\ph_L(z)=\ph(z/L)$. Then we have 
\[\var\l(\int \ph_L \ d\ze_\a\r) \le C_{\ph,\a} \int_0^{\infty} \!\!\! \int_0^{\infty} |\L \ph(r) \L \ph(s) | \exp(-\a L^{\frac{2}{\a}}(r^{\frac{1}{\a}}-s^{\frac{1}{\a}})^2)L^{-\frac{1}{\a}}(rs)^{1-\frac{1}{2\a}}dr ds. \] 
\end{lemma}

\begin{proof}[Proof of Lemma \ref{lem:varlinstat}]
 The random analytic function $\fa$ is a complex Gaussian field on $\C$ whose covariance kernel of which is given by
 \begin{equation} \label{eq:series_exp}
 K_{\a}(z,w):=\E[\fa(z)\overline{\fa(w)}]=\sum_{j=0}^{\infty} \frac{(z\ol{w})^j}{\ja}.
 \end{equation}
 The normalized random function $\hat{\fa}(z):=\fa(z)/\sqrt{K_{\a}(z,z)}$ has covariance
 \[\t_{\a}(z,w):=\E[\hat{\fa}(z)\overline{\fa(w)}]=\frac{K_{\a}(z,w)}{\sqrt{K_{\a}(z,z)K_{\a}(w,w)}}.\]
The variance of linear statistics will be expressed in terms of  the quantity
 \begin{align*}
 \psi_{\a}(z,w)&:=\E[\log|\hat{\fa}(z)|\cdot \log|\hat{\fa}(w)|] = \sum_{j=1}^{\infty} \frac{|\t_{\a}(z,w)|^{2j}}{4 j^2}.
 \end{align*}
It may be shown (c.f. Lemma \ref{lem:variance} at the end of this section)  that
\[
\var\l(\int \ph_L \ d\ze_\a\r) = \frac{1}{4 \pi^2} \int\!\!\! \int \L \ph_L(z) \L \ph_L(w) \psi_{\a}(z,w) dm(z)dm(w);
\]
Here and everywhere in the paper, we use $dm(z)$ to denote Lebesgue measure.

Since $|\t_{\a}(z,w)| \le 1$, we have the  comparison $|\t_{\a}(z,w)|^2 \le \psi_{\a}(z,w) \le 2 |\t_{\a}(z,w)|^2$ (because $\sum j^{-2}\le 2$). This implies that
\begin{align}
\var\l(\int \ph_L \ d\ze_\a \r) &\le \frac{2}{4 \pi^2}  \int\!\!\! \int |\L \ph_L(z) \L \ph_L(w)|\cdot |\t_{\a}(z,w)|^2 dm(z)dm(w) \nonumber \\
&= \frac{2}{4 \pi^2}  \int \!\!\! \int |\L \ph(z) \L \ph(w)|\cdot |\t_{\a}(Lz,Lw)|^2 dm(z)dm(w). \label{eq:bdforlinearstatvar}
 \end{align}
In the last step we made a simple change of variables.

Recall that we are considering  functions $\ph$ for which \begin{inparaenum}[(a)] \item $|\Delta \ph|$ is radial and \item $\Delta \ph$ vanishes in a neighbourhood of the origin. \end{inparaenum}    Further, recall the definition of $K_\a(z,w)$, and notice that $K_\a(z,z)$ and $\theta_\a(w,w)$ are also radial functions. Now consider integrating over the angular components of $z$ and $w$ (expressed in polar co-ordinates) on the right hand side of \eqref{eq:bdforlinearstatvar}. In other words, we write $z=re^{i \mu}$ and $w= s e^{i \nu}$, and we first perform the integration on the right hand side of \eqref{eq:bdforlinearstatvar} with respect to the variables $\mu$ and $\nu$. Due to the radial nature of the functions $\ph$ and $K_\a(z,z), K_\a(w,w)$ above, we may deduce that 
\begin{align*}
& \var\l(\int \ph_L \ d\ze_\a \r)  \\ \le & \frac{2}{4 \pi^2} 
  \int \!\!\! \int \frac{|\L \ph(r) \L \ph(s)|} {|K_{\a}(Lre^{i\mu},Lre^{i\mu})|^2 |K_{\a}(Lse^{i\nu},Lse^{i\nu})|^2} \cdot \l( \int \int |K_\a(Lre^{i\mu},Lse^{i\nu})|^2 d\mu d\nu   \r)  rdr sds. \numberthis \label{eq:angular}
\end{align*}  
We perform the integration over the angular variables $\mu,\nu$ in \eqref{eq:angular} by expanding $K_\a(Lre^{i\mu},Lse^{i\nu})$ in its series expansion \eqref{eq:series_exp}.
 
  Set \[g(t,\beta)= \sum_{k=0}^{\infty} \frac{t^{k}}{(k!)^{\beta}},\]
  and denote $a_k= \frac{t^{k}}{(k!)^{\beta}}$.  
Thus, from \eqref{eq:angular}, via integration over angular variables as above, we may arrive at
  \begin{align*}
 \var\l(\int \ph_L \ d\ze_\a \r) &\le 2  \int_0^{\infty} \!\!\! \int_0^{\infty} |\L \ph(r) \L \ph(s) | \frac{g(L^4r^2s^2,2\a)}{g(L^2r^2,\a)g(L^2s^2,\a)} rs dr ds.
 \end{align*}
 
 The ratio of two successive terms in the expansion of $g(t,\beta)$ is \[\frac{a_{k+1}}{a_k}=\frac{t}{(k+1)^{\beta}}.\] This implies that $a_k$ is maximised when $t=k^{\beta}$, in other words the maximiser is $k_* = t^{1/\beta} $ (strictly speaking $\lfloor t^{1/\beta}\rfloor$, but for ease of notation we ignore the symbol for integer part). The maximal term, via Stirling's approximation for the factorial, is \[ a_{k_*}= \exp\l(t^{1/\beta}\log t - \beta \frac{(t^{1/\beta} + \frac{1}{2}) \log t}{\beta} + \beta t^{1/\beta} \r) = t^{-\frac12}e^{\beta t^{1/\beta}} .\]

 For a positive integer $j$, we have \[a_{k_*+j}=\frac{t^{k_*+j}}{[(k_*+j)!]^{\beta}}= a_{k_*} \frac{t^j}{k_* ^{j \beta}} \prod_{i=1}^j\l( 1+ \frac{i}{k_*} \r)^{-\beta} \le a_{k_*} \prod_{i=1}^j\l( 1+ \frac{i}{k_*} \r)^{-\beta}.  \]
 Each term $(1+\frac{i}{k_{*}})^{-\beta}$ in the product is less than $1$. Further, for $i>\frac12 k_{*}$ this term is at most $(2/3)^{\beta}$. For  $1\le i\le \frac12 k_{*}$ by virtue of the inequality $e^{x/2}\le 1+x\le e^{x}$ for $0<x<\frac12$, the same term is between $e^{-\beta i/k_{*}}$ and $e^{-\beta i/2k_{*}}$. Putting all this together, we have
 \begin{align*}
  \frac{a_{k_{*}+j}}{a_{k_{*}}} \le  \l(\frac{2}{3}\r)^{j-\frac12 k_{*}} \mb{ if }j>\frac12 k_{*}, \\
e^{-\beta j^{2}/2k_{*}}  \le \frac{a_{k_{*}+j}}{a_{k_{*}}} \le e^{-\beta j^{2}/4k_{*}} \mb{ if }j\le \frac12 k_{*}.
 \end{align*}
Similarly,  $a_{k_{*}-j}=a_{k_{*}}\prod_{i=0}^{j-1}(1-\frac{i}{k_{*}})^{\beta}$ from which by analogous steps (use $e^{-2x}\le 1-x\le e^{-x}$ for $0<x<\frac12$) we get
\begin{align*}
  \frac{a_{k_{*}-j}}{a_{k_{*}}} \le  \l(\frac{1}{2}\r)^{j-\frac12 k_{*}} \mb{ if }j>\frac12 k_{*}, \\
e^{-\beta j^{2}/k_{*}}  \le \frac{a_{k_{*}+j}}{a_{k_{*}}} \le e^{-\beta j^{2}/2k_{*}} \mb{ if }j\le \frac12 k_{*}.
 \end{align*}
 Thus, for $j \ge \sqrt{k_*/c\beta}$, the quantity $a_{k_* \pm j}/a_{k_*}$ decays exponentially. The upshot of this is that $g(t,\beta)/a_{k_*}\sqrt{k_*}$ is bounded between positive quantities that depend only on $\beta$. Therefore, the ratio between the quantitites
 \begin{align*}
 \frac{g(L^4r^2s^2,2\a)}{g(L^2r^2,\a)g(L^2s^2,\a)}\end{align*}
  and
   \begin{align*}&&\frac{\l(\frac{e^{2\a(L^4r^2s^2)^{1/2\a}}}{\sqrt{L^4r^2s^2}}(L^4r^2s^2)^{1/4\a}\r)}{\l(\frac{e^{\a(L^2r^2)^{1/\a}}}{\sqrt{L^2r^2}} \cdot (L^2r^2)^{1/2\a}\frac{e^{\a(L^2s^2)^{1/\a}}}{\sqrt{L^2s^2}} \cdot (L^2s^2)^{1/2\a}\r)}  =\frac{\exp(-\a L^{2/\a}(r^{1/\a}-s^{1/\a})^2)}{L^{1/\a}r^{1/2\a}s^{1/2\a}} \end{align*}
  is bounded between two constants $c_{\beta}$ and $C_{\beta}$.

 As a result,
 \begin{align}
 &\int_0^{\infty}\!\!\! \int_0^{\infty} |\L \ph(r) \L \ph(s) | \frac{g(L^4r^2s^2,2\a)}{g(L^2r^2,\a)g(L^2s^2,\a)} rs dr ds \label{eq:varianceboundalphagaf} \\
 &  \le C_{\beta} \int_0^{\infty} \!\!\! \int_0^{\infty} |\L \ph(r) \L \ph(s) | \exp(-\a L^{2/\a}(r^{1/\a}-s^{1/\a})^2)L^{-1/\a}r^{1-1/2\a}s^{1-1/2\a} dr ds,
\nonumber
 \end{align}
The last integral is well defined since $\L \ph (r) \equiv 0$ for $r$ close to $0$. 

This completes the proof of Lemma \ref{lem:varlinstat}.
\end{proof}


We now move on to complete the proof Lemma \ref{rigidity} (and hence  the proof of Theorem \ref{hierig}).

\begin{proof}[Proof of Lemma \ref{rigidity}]
We apply the bound to  a specific family of radial functions. Given any $r_{0}>1$ and any $\eps>0$, let $\ph_{\eps}$ be a radial function such that
\begin{enumerate}
\item $\ph_{\eps}(r)=1$ for $r<2r_{0}$ and $\ph_{\eps}(r)=0$ for $r>2r_{0}e^{2/\eps}$,
\item $|\ph_\eps'(r)|\le \frac{\eps}{2r}$ and $|\ph_\eps''(r)|\le\frac{\eps}{2r^2}$ for all $r$.
\end{enumerate}

Such a function can be constructed by starting with the harmonic function $(\log|z|-\log a)/(\log b-\log a)$ for some $2r_{0}<a<b<2r_{0}e$ and  convolving it with  the indicator function of a disk $D(0,\delta)$ for a small enough $\del>0$.

Consider the function $\Phi^\eps(z)=z^k\ph_\eps(z)$. Observe that $\Phi^\eps$ satisfies the conditions (a) and (b) above. As a result,  $\Phi^\eps$ satisfies the inequality \eqref{eq:varianceboundalphagaf}. Hence we have
\begin{equation} \label{eq:varianceboundalphagaf_intermediate}
\mathrm{Var}\l( \int L^k \Phi^\eps_L d \F_\a \r) \le L^{2k - \frac{2}{\a}} \int |\Delta \Phi^\eps (r)|^2 r^{3-\frac{2}{\a}} dr + K(\Phi^\eps)L^{2k}e^{-cL^{2/\a}},
\end{equation}
where $K(\Phi^\eps)$ is a positive quantity that depends only on $\Phi^\eps$ (and is independent of $L$).
We briefly mention here as to how we obtain the above inequality. To this end, on the right hand side of \eqref{eq:varianceboundalphagaf} we carry out integration with respect to the variable $s$. This integration can be decomposed into two regimes of the variable $s$ : namely, $\{s : r^{1/\a}-s^{1/\a} \le 1  \}$ and $\{s : r^{1/\a}-s^{1/\a} > 1  \}$. The two terms on the right hand side of \eqref{eq:varianceboundalphagaf_intermediate} are obtained from these two regimes of integration respectively.

Observe that $ 2k - 2/\a \le 0 $, so $ L^{2k - \frac{2}{\a}} \le 1$.

We now claim that $ \int |\Delta \Phi^\eps (r)|^2 r^{3-\frac{2}{\a}} dr = O(\eps)$. To see this, applying Leibniz rule to compute the Laplacian $|\Delta \Phi^\eps (r)|$ is supported on the set $2r_0 \le r \le 2r_0\exp(2/\eps)$ , and is $O(\eps r^{k-2})$ on this interval. This implies that
\[\int |\Delta \Phi^\eps (r)|^2 r^{3-\frac{2}{\a}} dr \le C_1 \eps^2 \int_{2r_0}^{2r_0\exp(2/\eps)} r^{2k-4+3-\frac{2}{\a}} dr \le C_2 \eps^2 \int_{2r_0}^{2r_0\exp(2/\eps)} \frac{dr}{r} \le C_3 \eps. \]

Now, given $\delta>0$, first choose $\eps$ such that  $\int |\Delta \Phi^\eps (r)|^2 r^{3-\frac{2}{\a}} dr \le \delta/2$; based on that choose $L$ large enough such that $K(\Phi^\eps)L^{2k}e^{-cL^{2/\a}} \le \delta/2$. 
This enables us to invoke Theorem 6.1 from \cite{GP} (using the function $L^k\Phi_L^\eps(z)$ as defined above; for a statement see Theorem \ref{thm:GP} in Section \ref{sec:thiswillberewritten} ) and conclude that there is a measurable function $S_k:\mathcal{S}_{D^c} \to \C$ such that $\sum_{z \in \ze_\a \cap D} z^k = S_k(\ze_\a \cap D^c)$ a.s. This completes the proof of  Lemma \ref{rigidity} (and hence the proof of Theorem \ref{hierig}).
\end{proof}

We end this section with a technical lemma that provides us with an expression for $\var\l(\int \ph_L \ d\ze_\a\r)$.

\begin{lemma} \label{lem:variance}
Let $\ph$ be a $C_c^2$ function and $L>0$. Then we have
\[
\var\l(\int \ph_L \ d\ze_\a\r) = \frac{1}{4 \pi^2}  \int\!\!\! \int \L \ph_L(z) \L \ph_L(w) \psi_{\a}(z,w) dm(z)dm(w).
\]
\end{lemma}

\begin{proof}
We consider the centered random variable 
\[ \left( \int \ph_L \ d\ze_\a \right) - \E \left[\int \ph_L \ d\ze_\a \right].  \]
We may write, using the Edelman-Kostlan formula (c.f. Eq. (2.4.8) in \cite{HKPV} Chap 2; see also Eq. (3.5.2) in \cite{HKPV} Chap. 3), that 
\begin{equation} 
\left( \int \ph_L \ d\ze_\a \right) - \E \left[\int \ph_L \ d\ze_\a \right] = \frac{1}{2\pi}\int \L \phi_L(z) \log |\hat{\fa}(z)| dm(z).
\end{equation}
Thus, 
\begin{align*} 
& \var\l(\int \ph_L \ d\ze_\a\r)  \\ 
=  & \var\l(\frac{1}{2\pi}\int \L \phi_L(z) \log |\hat{\fa}(z)| dm(z)\r) \\
= & \frac{1}{4\pi^2} \int \int \L \phi_L(z) \L \ph_L(w) \mathrm{Cov}\l( \log |\hat{\fa}(z)|, \log |\hat{\fa}(w)| \r) dm(z). \numberthis \label{eq:intermediate}
\end{align*}
We now use the following result ( Lemma~3.5.2  in \cite{HKPV}) : if $a,b$ are complex Gaussians with $\E[|a|^2]=\E[|b|^2]=1$ and $\E[a \overline{b}]=\theta$, then 
  \[    \mathrm{Cov}\l( \log |a|, \log |b| \r)   =   \sum_{m=1}^\infty \frac{|\theta|^{2m}}{4m^2}.  \] Applying this to \eqref{eq:intermediate}, we obtain the desired result.
\end{proof}

\section{Proof of Theorem~\ref{detrig}: Non-projection determinantal processes are tolerant}
We first sketch the basic idea, and then move on to the proof. Consider a  kernel of the form $K(x,y)=\sum_{j\ge 1}\la_{j}\ph_{j}(x)\bar{\ph}_{j}(y)$ where $0\le \la_{j}\le 1$ and $\ph_{j}$ are orthonormal in $L^{2}(\mu)$. Let $I_{j}$ be independent Bernoulli random variables with $\P\{I_{j}=1\}=\la_{j}$, and define $K_{I}(x,y)=\sum_{j\ge 1}I_{j}\ph_{j}(x)\bar{\ph}_{j}(y)$. If we sample $I_{j}$s first and then a determinantal point process with kernel $K_{I}$, the resulting process is precisely the determinantal point process with kernel $K$ (see Thereom~4.5.3 in \cite{HKPV}). In particular, if $0<\la_{j}<1$ for at least one $j$, then the  number of points in the point process is a non-constant random variable,  or it is almost surely infinite. For the special case when the expected number of points $\sum_j\lambda_j<\infty$, this shows that rigidity fails even when $D=\Xi$ (also see   Lemma~\ref{lem:inclusion}).


In general, $K$ (rather, its associated operator $\mathcal K$) may have a continuous spectrum and hence we cannot write an expansion as above. However, we can write the kernel $K$ as a convex combination of two projection kernels and extend the above idea to show that the determinantal process with kernel $K$ is a mixture of two determinantal processes. From this, we can deduce that the given determinantal process does not have rigidity of numbers. The essential observation is that the proof of Theorem~4.5.3 in \cite{HKPV} referred to above, does not require $\ph_{j}$s to be orthogonal.


%

\begin{proof}[Proof of Theorem \ref{detrig}]
Let $K$ be a locally trace-class, Hermitian, contraction kernel on $L^{2}(\Xi,\mu)$ which is not a projection. Then, the corresponding integral operator  $\mathcal K=\int_{0}^{1}\la dE(\la)$ where the projection-valued measure $E$ has the property that $E[\del,1-\del]\not=0$ for some $\del>0$. 

Let $f$ be in the range of $E[\del,1-\del]$ with $\|f\|^{2}=\del$ (norms in $L^{2}(\mu)$). Let 
\begin{align*}
L_{0}(x,y)=K(x,y)-\half f(x)\bar{f}(y)\;\;\;\mb{ and } \;\;\; L_{1}(x,y)=K(x,y)+ \half f(x)\bar{f}(y).
\end{align*} 
Let $\mathcal L_{0}$ and $\mathcal L_{1}$ be the associated operators.

Since $\mathcal K$ is locally of trace-class and Hermitian and $f\in L^{2}(\mu)$, it follows that $\mathcal L_{0}$ and $\mathcal L_{1}$ are also locally of trace-class and Hermitian. Further, by the choice of $f$, it follows that they are contraction kernels on $L^{2}(E,\mu)$, i.e., $0\le \mathcal L_{0}\le I$ and $0\le \mathcal L_{1}\le I$. To see this, let $P_{\del}$ denote the spectral projection $E[\del,1-\del]$ and let $f\otimes f^{*}$ denote the operator on $L^{2}(\mu)$ that maps $\psi$ to $\langle \psi,f \rangle f$. Then $f\otimes f^{*}\le \del P_{\del}$ (in the positive definite order) and $\del P_{\del}\le \mathcal K\le (1-\del)P_{\del}$. Therefore, 
\begin{align*}
\mathcal L_{0}=(K-\frac12 f\otimes f^{*})\ge\frac12 \del P_{\del}, \\
 \mathcal L_{1}=(K+\frac12 f\otimes f^{*})\le (1-\frac12 \del)P_{\del}.
\end{align*}
This shows that $\mathcal L_{0}$ and $\mathcal L_{1}$ are positive contractions and hence, by the theorem of Macchi and Soshnikov (see the discussion preceding the statement of Theorem~\ref{detrig}), there exist determinantal processes $\Pi_{0}$ and $\Pi_{1}$ with kernels $L_{0}$ and $L_{1}$, respectively.

By a result tracing back to Lyons, we may construct $\Pi_{0}$ and $\Pi_{1}$ on the same probability space such that every point of  $\Pi_{0}$ is also a point of $\Pi_{1}$, along with a  Bernoulli$(1/2)$ random variable $\xi$ that is independent of $\Pi_{0}$ and $\Pi_{1}$.  For Lyons' original result (applicable to the case when $\Xi$ is a countable set),  see Theorem~6.2 of \cite{Ly}. For  more recent versions that address the general case, we refer the reader to  \cite{Tripathi, Lyons-ICM}. The  result in the general setting can, in fact,  be obtained from Lyon's original result for the countable case via the so-called \textit{Goldman's transference principle} \cite{Go}.

 As $\mathcal L_{1}-\mathcal L_{0}=f\otimes f^{*}$ has rank one, in view of the above coupling there is at most one point in $\Pi_{1}$ that is not in $\Pi_{0}$. Define
\[
\Pi=\begin{cases}
\Pi_{0} &\mb{ if }\xi=0, \\ \Pi_{1} &\mb{ if }\xi=1.
\end{cases}
\]
We claim that $\Pi$ is a determinantal point process with kernel $K$. To show this, fix $k\ge 1$ and $x_{1},\ldots ,x_{k}\in E$ and let  $M=\l(L_{0}(x_{i},x_{j})\r)_{i,j\le k}$ and $\uu^{t}=(f(x_{1}),\ldots ,f(x_{k}))$. Then, $\l(L_{1}(x_{i},x_{j})\r)_{i,j\le k}=M+\uu \uu^{*}$. Hence, $\Pi$ has $k$-point intensity given by
\begin{align*}
p_{k}(x_{1},\ldots ,x_{k}) &= \half \det\l( L_{1}(x_{i},x_{j})\r)_{i,j\le k} + \half \det\l( L_{0}(x_{i},x_{j})\r)_{i,j\le k} \\
&= \frac12 \det(M+\uu \uu^{*})+\frac12 \det(M) \\
&= \half \det(M)\l(1+ \uu^{*} M^{-1} \uu \r) + \half \det(M)\\
&=  \det(M)\l(1+\half \uu^{*} M^{-1} \uu \r) \\
&= \det\l(M+\half \uu \uu^{*}\r)
\end{align*}
which is precisely  $\det(K(x_{i},x_{j}))_{i,j\le k}$. This proves the claim.

Now, there exists a pre-compact set $D$ such that $\P\{\Pi_{1}\setminus \Pi_{0} \subseteq D\}>0$. On this event, $\Pi_{0}, \Pi_{1},\Pi$, all agree on $D^{c}$ but inside $D$, $\Pi_{1}$ has one point more than $\Pi_{0}$. Condition on the configuration $\Pi_{D^{c}}$. Depending on the value of $\xi$, the number of points of $\Pi$ inside $D$ can take two different values, each with positive probability (for a positive probability of configurations outside $D$).  This proves that $\Pi$ cannot have rigidity of numbers.
\end{proof}

\section{Determinantal processes with radial intensity measures}
\label{radialmeasures}

In this section, we will prove Theorems \ref{compact} and \ref{dpprt}. We will take up these proofs in two separate subsections. We recall the relevant notations from Section \ref{introduction}, and introduce some additional notation by setting $\Pi_n$ to be the d.p.p. generated by the $(K_n,\gamma)$, where $K_n$ is the truncated kernel given by \[K_n(z,w)=\sum_{j=0}^{n} c_j z^j\ol{w}^j.\] In our further considerations, we will use the concept of Palm measures. Roughly speaking, the Palm measure of a point process at the origin is the law of point process conditioned to contain the origin. The precise definition is given below. For a more detailed treatment, we refer the reader to chapter~10 of \cite{Ka} or chapter~3  of \cite{bbkbook}. 

\begin{definition} \label{def:Palm}
Let $\Pi$ be a point process on the space $X$,  let $\mathcal{S}$ denote the Polish space of all locally finite point configurations on $X$,  $\mathcal{B}(X)$ denote the Borel sigma field on $\mathcal{S}$ and $\rho_1$ denote the one-point intensity measure of $\Pi$. Let $F : X \times \mathcal{S} \mapsto \R$ be a test function. Then, for $s \in X$, the Palm measure  $\Pi_s$ at $s$ is defined by  
$\E_\Pi \l[ \sum_{s \in \Pi} F(s,\Pi) \r] = \int_X \E_{\Pi_s} \l[  F(s,\Pi_s) \r] \mathrm{d}\rho_1(s)$.
\end{definition}

For a determinantal process  with kernel $K$, and any $s\in \Xi$ with $K(s,s)>0$, the Palm measure at $s$ exists and is again a determinantal process with kernel
\[
K_s(x,y)=K(x,y)-\frac{K(x,s)K(s,y)}{K(s,s)}.
\]
See Theorem~1.7 in \cite{shiraitakahashi}.

\subsection{Proof of Theorem \ref{compact}}



\begin{proof}
 We will show that $\Pi$ and the Palm measure of $\Pi$ at the origin can be coupled so that they are equal with positive probability. This will show that, with positive probability, a point can suitably be removed or added to $\Pi$ without detection, establishing that there cannot be rigidity of the number of points. In fact, we show that the Palm measure $\wt{\Pi}_n$ of $\Pi_n$ at the origin and $\Pi_{n-1}$ can be coupled so that they are equal with a positive probability, that stays bounded away from 0 as $n \to \infty$. Then, using the fact that $\Pi_n \to \Pi$ and $\wt{\Pi}_n \to \wt{\Pi}$ weakly, we can draw the desired conclusion about $\Pi$.

 Note that the process $\Pi_n$ is determinantal with kernel $K_n$ and hence contains $n+1$ points,  almost surely. The joint density of  the points of $\Pi_n$ (when placed in uniform random order)  is given by \[p_n(z_1,\cdots,z_{n+1})=C c_0 c_1 \cdots c_n |\Delta(z_1,\cdots,z_{n+1})|^2,\] where $\Delta(z_1,\cdots,z_{n+1})$ is, as usual, the Vandermonde determinant formed by $z_1,\cdots,z_{n+1}$. The Palm measure $\wt{\Pi}_n$ of $\Pi_n$ at the origin is a determinantal point process with background measure $\gamma$ and kernel given by $\wt{K}_n(z,w)=K_n(z,w)-c_0$. Thus, $\wt{\Pi}_n$ has a density given by $\wt{p}_n(z_1,\cdots,z_n)=C c_1 \cdots c_n|z_1|^2\cdots|z_n|^2  |\Delta(z_1,\cdots,z_n)|^2$. Then the Radon Nikodym derivative of $\wt{\Pi}_n$ w.r.t. $\Pi_{n-1}$ is given by \[\t_n(z_1,\cdots,z_n)=\frac{\wt{p}_n(z_1,\cdots,z_n)}{p_{n-1}(z_1,\cdots,z_n)}= c_n |z_1|^2 \cdots |z_n|^2.\]

 The maximal probability, under any coupling of $\wt{\Pi}_n$ and $\Pi_{n-1}$, that the coupled random variables are equal is $\E_{\Pi_{n-1}}\l[ \t_n(z_1,\cdots,z_n) \wedge 1  \r]$. Therefore, our goal is to show that $\E_{\Pi_{n-1}}\l[ \t_n(z_1,\cdots,z_n) \wedge 1  \r]$ is bounded away from 0 as $n \to \infty$. One can then take a subsequential weak limit of the finite dimensional couplings and obtain a coupling of $\Pi$ and $\wt{\Pi}$ with the desired property. We will show that, under the growth assumption on the moments as described in the statement of the theorem, $  \t_n(z_1,\cdots,z_n)= c_n |z_1|^2 \cdots |z_n|^2 \to Z$ a.s., where $Z>0$ with positive probability. 
 
To this end, we consider $W$ to be a random variable distributed according to the probability measure $\mu(\cdot)/\mu(\C)$. Then we set $\Gamma_1 = |W|^2$, and each $\Gamma_k$ is a size-biased version of $\Gamma_{k-1}$, i.e. the corresponding probability measures satisfy $d \Gamma_k = \frac{c_k}{c_{k-1}} d \Gamma_{k-1}$. Theorem Theorem~4.7.1 in \cite{HKPV} essentially states that, if $(z_1,\ldots,z_n) \sim \Pi_{n-1}$, then  $\{|z_i|^2\}_{i=1}^n$ has the same distribution as $\{|\Gamma_i|\}_{i=1}^n$, where the $\Gamma_i$-s are independent copies of the random variables described above.

Thus, observe that we can write, using Theorem~4.7.1 in \cite{HKPV},
 \[c_n |z_1|^2 \cdots |z_n|^2 =  c_0 \prod_{i=1}^n \frac{|z_i|^2}{(c_{i-1}/c_i)} =  c_0 \prod_{j=1}^n\frac{\Gamma_j}{\mu_j}= c_0  \prod_{j=1}^n\l(1+\l(\frac{\Gamma_j}{\mu_j}-1\r)\r),\] where $\Gamma_j$-s are independent copies of the size biased random variables. discussed above. 
 
 Thus, it suffices to show that $\l|\sum_{j=1}^{\infty}\l(\frac{\Gamma_j}{\mu_j}-1\r) \r|<\infty$. By the Kolmogorov's Two Series Theorem (c.f. \cite{Durrett} Chap. 2), this will be true if  $\sum_{j=1}^{\infty}\mathrm{Var}\l(\frac{\Gamma_j}{\mu_j}-1\r) <\infty$. This is exactly the condition \[\sum_{j=1}^{\infty}\l( \frac{\mu_{j+1}}{\mu_j} -1 \r) < \infty.\] This is precisely growth criterion laid out in the statement of the theorem.
 \end{proof}

\subsection{Proof of Theorem \ref{dpprt}}
\label{dpprtproof}



We begin by recalling that   $\Pi$ is the d.p.p. generated by the $(K,\gamma)$, where $K$ is the  kernel given by \[K(z,w)=\sum_{j=0}^{\infty} c_j z^j\ol{w}^j.\] On the other hand,
$\Pi_n$ is the d.p.p. generated by the $(K_n,\gamma)$, where $K_n$ is the truncated kernel given by \[K_n(z,w)=\sum_{j=0}^{n} c_j z^j\ol{w}^j.\] The sequence of d.p.p.-s $\{\Pi_n\}_{n \ge 0}$ converges to the d.p.p. $\Pi$ in the sense of convergence of point processes (\cite{DV}).

We will study the variance of linear statistics of $\Pi$ and $\Pi_n$, and show that under certain conditions on the coefficients, the variance of linear statistics remains bounded even as the scaling factor goes to infinity. 
Invoking a result in \cite{GP} (c.f. Theorem 6.1 therein; for a statement see Theorem \ref{thm:GP} in Section \ref{sec:thiswillberewritten}), this would be enough to guarantee the rigidity of the number of points of $\Pi$ in any bounded open set. For a more detailed discussion on this matter, we refer the reader to the Proof of Theorem \ref{dpprt} later in this section. 

To prove Theorem \ref{dpprt}, we will first state and prove an estimate on the variance of linear statistics of $\Pi$.

\begin{proposition}
\label{est}
Let $\varphi$ be a compactly supported Lipschitz function, supported inside the disk $B(0;r)$ with Lipschitz constant $\kappa(\varphi)$. Let $\varphi_R(z):=\varphi(z/R)$. Suppose  that the growth conditions in the statement of Theorem \ref{dpprt} hold.

Then we have 
\begin{equation} \label{estimate} 
\mathrm{Var} \l(\int\varphi_R(z)\, d[\Pi_n](z)\r)\le B \kappa(\varphi)^2 + C \|\varphi\|^2_\infty \rho(R),
\end{equation} 
where $B,C$ are positive numbers and $\rho(R)$ is a quantity that $\to 0$ as $R \to \infty$, uniformly in $n$ and $\varphi$.
The same conclusion holds for $\Pi$ in place of $\Pi_n$.
\end{proposition}

To prove Proposition \ref{est}, we will make use of a general fact about determinantal point processes:
\begin{lemma}
\label{varlinst}
Let $\Pi$ be a determinantal point process with Hermitian kernel $K$. Let $K$ be a reproducing kernel with respect to its background measure $\gamma$, i.e., $K(x,y)= \int K(x,z)K(z,y)\, \mathrm{d}\gamma(z)$ for all $x,y$. Let $\varphi, \psi$ be  compactly supported continuous functions.  Then we have
\[\mathrm{Cov} \l( \int \varphi \, d\Pi, \int \psi \, d\Pi \r) = \frac{1}{2}\iint (\varphi(z) - \varphi (w)) \overline{(\psi(z)-\psi(w))}|K(z,w)|^2 \, \mathrm{d}\gamma(z) \, \mathrm{d}\gamma(w).\]
\end{lemma}

\begin{proof} [\textbf{Proof of Proposition \ref{est}}]
We give the proof when $r=1$, from here the general case is obtained by scaling, because any function $\varphi$ supported on $B(0;r)$ is equal to the scaling $\Phi_r$ of some function $\Phi$ supported on $B(0;1)$ and having the same continuity and differentiability properties as $\varphi$. Notice that $(\Phi_r)_R=\Phi_{rR}$, so the result for $\varphi$ can be deduced from the result for $\Phi$. In what follows we deal with $\Pi_n$, the result for $\Pi$ follows, for instance, from taking limits as $n \to \infty$ for the result for $\Pi_n$.

Using Lemma \ref{varlinst}, we have \[\text{Var} \l(\int \varphi_R(z)\, d[\Pi_n](z) \r)= \frac{1}{2} \iint {|\varphi_R(z)-\varphi_R(w)|^2|K_n(z,w)|^2\, \mathrm{d}\gamma(z)\, \mathrm{d}\gamma(w)}\]
where $\gamma$ is the background measure.
Now, \[|\varphi_R(z)-\varphi_R(w)|^2 = |\varphi(z/R)-\varphi(w/R)|^2 \le \frac{1}{R^2}\kappa(\varphi)^2|z-w|^2. \]
Therefore, it suffices to bound the integral $\int_{A(R)} |\varphi_R(z)-\varphi_R(w)|^2 |K_n(z,w)|^2 \, \mathrm{d}\gamma(z) \, \mathrm{d}\gamma(w)$ on the set \[A(R):= \{ (z,w) : \min\{|z|, |w|\} \le R \} \] because outside $A(R)$, we have $\varphi_R(z)=\varphi_R(w)=0$.  

Thus, we may write
\begin{align*}
& \int_{A(R)} |\varphi(z/R)-\varphi(w/R)|^2 |K_n(z,w)|^2 \, \mathrm{d}\gamma(z) \, \mathrm{d}\gamma(w) \\ 
\le & \frac{\kappa(\varphi)^2}{R^2} \int_{A_1(R)} |z-w|^2 |K_n(z,w)|^2 \, \mathrm{d}\gamma(z) \, \mathrm{d}\gamma(w) + \|\varphi\|_\infty^2 \int_{A_2(R)} |K_n(z,w)|^2 \, \mathrm{d}\gamma(z) \, \mathrm{d}\gamma(w) 
\numberthis \label{eq:A1-A2}
\end{align*}
where, for  $a \ge 2$, we set \[A_1(R)=\{ |z|\le aR, |w| \le aR \} \text{ and } A_2(R)=\{ |z|\le R, |w| \ge aR \} \cup \{ |w|\le R, |z| \ge aR \}. \]


In \eqref{eq:A1-A2}, the integral over $A_1(R)$ will provide an upper bound of $B \kappa(\varphi)^2$, whereas the integral over $A_2(R)$ will be controlled by a quantity of the form $\|\varphi\|^2_\infty \rho(R)$, where 
$\rho(R) \to 0$ as $R \to \infty$, uniformly in $n$ and $\varphi$. The bounds for these two integrals will be taken up in Propositions \ref{prop:A1} and \ref{prop:A2} respectively.

Combining \eqref{eq:A1-A2} with Propositions \ref{prop:A1} and \ref{prop:A2}, we obtain the desired estimate on the variance of linear statistics as stated in Proposition \ref{est}.
\end{proof}

\begin{proposition}  \label{prop:A1}
Consider the integral 
\[ I_1(R,\varphi) := \frac{\kappa(\varphi)^2}{R^2} \int_{A_1(R)} |z-w|^2 |K_n(z,w)|^2 \, \mathrm{d}\gamma(z) \, \mathrm{d}\gamma(w)\] as in \eqref{eq:A1-A2}.
Then we have
\[ I_1(R,\varphi) \le  B \kappa(\varphi)^2 \]
for some positive number $B$.
\end{proposition}

\begin{proof}
We begin with
\begin{align*}
&\int_{A_1(R)} |z-w|^2 |K_n(z,w)|^2\, \mathrm{d}\gamma(z) \mathrm{d}\gamma(w)
\\& = \int \bigg( |z|^2-z\bar{w}-\bar{z}w+|w|^2 \bigg) \l(\sum_{j=0}^{n-1}c_j( z\bar{w})^j\r) \l(\sum_{j=0}^{n-1} c_j( \bar{z}w)^j\r) \, \mathrm{d}\gamma(z) \, \mathrm{d}\gamma(w). \numberthis \label{eq:integral}
\end{align*}
Now, we integrate the $|z-w|^2$ part term by term. Due to the radial symmetry of $\gamma$, only some specific terms from $|K_n(z,w)|^2$ contribute. For example, when we integrate the $|z|^2$ term in $|z-w|^2$, only the ${\displaystyle c_j (z\bar{w})^j c_j (\bar{z}w)^j}$,$0\le j\le n-1$ terms in the expanded expression for $|K_n(z,w)|^2$ contribute. When we integrate $z\bar{w}$, only the ${\displaystyle c_j(z\bar{w})^{j}c_{j+1} (\bar{z}w)^{j+1} }$,$0\le j\le{n-2}$ terms provide non-zero contributions. Due to symmetry between $z$ and $w$, it is enough to bound the contribution from $(|z|^2 - z \overline{w})$ by  $O(R^2)$, with the constant in the $O$ being independent of $n$.

In what follows, we will denote by $\mu$ the measure on $\R_+$ obtained by the push forward of the measure $\gamma$ from $\C$ to $\R_+$ under the map $z \mapsto |z|^2$.

$\bullet$ \underline{$|z|^2$ term: }

Let us introduce the change of variables  by $x=|z|^2$ and $y=|w|^2$. Then the contribution in the above integral \eqref{eq:integral} coming from \[ c_j (z\bar{w})^jc_j (\bar{z}w)^j\] can be written as \[ \int_0^{a^2R^2} \int_0^{a^2R^2} c_j x^{j+1} c_j y^{j} d\mu(x)d\mu(y). \] So, the total contribution due to all such terms, ranging from $j=0,\cdots,n-1$ is \[ \sum_{j=0}^{n-1}  \l( \int_0^{a^2 R^2} c_j x^{j+1} d\mu(x) \r) \l( \int_0^{a^2 R^2} c_{j+1} y^{j}d\mu(y) \r). \]

$\bullet$
 \underline{$z\bar{w}$ term: }

As above, the contribution coming from the \[ c_j(z\bar{w})^{j}c_{j+1}(\bar{z}w)^{j+1} \] is given by \[ \int_0^{a^2R^2} \int_0^{a^2R^2}c_j x^{j+1} c_{j+1} y^{j+1}d\mu(x)d\mu(y).  \] Therefore the total contribution from $0\le j\le{n-2}$ is
\[\sum_{j=0}^{n-2} \l( \int_0^{a^2 R^2} c_j  x^{j+1} d\mu(x) \r) \l( \int_0^{a^2R^2} c_{j+1} y^{j+1} d\mu(y) \r).\]

Recall that ${\displaystyle c_j x^jd\mu(x)}$ is a probability density, denote the corresponding random variable by $\Gamma_{j+1}$ (it may be seen that the random variable $\Gamma_j$ defined here, and that introduced in the proof of Theorem \ref{compact} are the same, using Theorem 4.7.1 in \cite{HKPV}; however, we do not require this fact for our arguments in this subsection). 

The contribution due to the $|z|^2$ term is ${ \displaystyle \sum_{j=0}^{n-1} \E[\Gamma_{j+1}{1\hskip-4pt{\rm 1}}_{(\Gamma_{j+1}\le a^2R^2)}]\P[\Gamma_{j+1}\le a^2R^2]}$ and that due to the $z\bar{w}$ term is
${ \displaystyle \sum_{j=0}^{n-2} \E[\Gamma_{j+1}{1\hskip-4pt{\rm 1}}_{(\Gamma_{j+1}\le a^2R^2)}]\P[\Gamma_{j+2}\le a^2R^2] }$.

$\bullet$ \underline{Combining the terms:}
The difference between the above two terms can be written as:
\begin{equation}
 \label{qty}
 \E[\Gamma_{n}{1\hskip-4pt{\rm 1}}_{(\Gamma_{n}\le a^2R^2)}]\P[\Gamma_{n}\le a^2R^2]+\sum_{j=1}^{n-1}\E[\Gamma_j{1\hskip-4pt{\rm 1}}{(\Gamma_j\le a^2R^2)}]\bigg(\P[\Gamma_j\le a^2R^2]-\P[\Gamma_{j+1}\le a^2R^2] \bigg).
\end{equation}
All the expectations in the above are $\le a^2R^2$, and $\P[\Gamma_j\le a^2R^2] \ge \P[\Gamma_{j+1}\le a^2R^2]$ because $\Gamma_{j+1}$ stochastically dominates $\Gamma_j$ (this is true because  a random variable is stochastically dominated by its size biasing). 

Therefore the absolute value of \eqref{qty}, by a telescopic sum, is bounded above by 
\begin{align*}
& \E[\Gamma_{n}{1\hskip-4pt{\rm 1}}_{(\Gamma_{n}\le a^2R^2)}]\P[\Gamma_{n}\le a^2R^2]+\sum_{j=1}^{n-1}\E[\Gamma_j{1\hskip-4pt{\rm 1}}{(\Gamma_j\le a^2R^2)}]\bigg(\P[\Gamma_j\le a^2R^2]-\P[\Gamma_{j+1}\le a^2R^2] \bigg) \\
\le & a^2R^2 \P[\Gamma_1\le a^2R^2] \\
\le & a^2R^2. \numberthis \label{eq:qty-1}
\end{align*}


$\bullet$ Finally, we  recall the definition of $I_1(R,\varphi)$ and combine with \eqref{eq:qty-1} to complete the proof of the proposition.

\end{proof}

\begin{proposition} \label{prop:A2}
Let 
\[I_2(R,\varphi):=)\|\varphi\|_\infty^2 \int_{A_2(R)} |K_n(z,w)|^2 \, \mathrm{d}\gamma(z) \, \mathrm{d}\gamma(w)   \]
and set
\begin{equation} \label{suffcond} 
\rho(R):= \sum_{j=0}^n \P\l[\Gamma_j \le R^2\r] \P\l[ \Gamma_j \ge a^2R^2 \r]. 
\end{equation}
Then $I_2(R,\varphi) \le C \|\varphi\|^2_\infty \rho(R)$ for some positive number $C$ and, furthermore, $\rho(R) \to 0$ uniformly in $n, \varphi$ as $R \to \infty$.
\end{proposition}

\begin{proof}
We proceed with our analysis of $I_2(R,\varphi)$ as in the proof of Proposition \ref{prop:A1}. In this vein, we write, 
\begin{align}  
 \|\varphi\|_\infty^2 \int_{A_2(R)}&  |K_n(z,w)|^2\, \mathrm{d}\gamma(z) \mathrm{d}\gamma(w) \nonumber \\
 &= \|\varphi\|_\infty^2 \int_{A_2(R)} \l(\sum_{j=0}^{n-1}c_j( z\bar{w})^j\r) \l(\sum_{j=0}^{n-1} c_j( \bar{z}w)^j\r) \, \mathrm{d}\gamma(z) \, \mathrm{d}\gamma(w). \label{eq:integral-1}
\end{align}

Due to the radial symmetry of $\gamma$, only some specific terms from $|K_n(z,w)|^2$ contribute to the above integral. To be more specific, only the ${\displaystyle c_j (z\bar{w})^j c_j (\bar{z}w)^j}$,$0\le j\le n-1$ terms in the expanded expression for $|K_n(z,w)|^2$ contribute. 


As in the proof of Proposition \ref{prop:A1}, we will denote by $\mu$ the measure on $\R_+$ obtained by the push forward of the measure $\gamma$ from $\C$ to $\R_+$ under the map $z \mapsto |z|^2$.

Thus, we may rewrite the right hand side of \eqref{eq:integral-1} as 
\begin{align*}  
& \int \l(\sum_{j=0}^{n-1}c_j( z\bar{w})^j\r) \l(\sum_{j=0}^{n-1} c_j( \bar{z}w)^j\r) \, \mathrm{d}\gamma(z) \, \mathrm{d}\gamma(w) \\
= &  \sum_{j=0}^{n-1} \l( \int_0^{ R^2} c_j  x^j d\mu(x) \r) \l( \int_{a^2R^2}^{\infty} c_{j} y^{j} d\mu(y) \r) + \sum_{j=0}^{n-1} \l(   \int_{a^2R^2}^{\infty} c_j  x^j d\mu(x) \r) \l( \int_0^{ R^2} c_{j} y^{j} d\mu(y) \r)
\numberthis \label{eq:integral-2}
\end{align*}
Due to reasons of symmetry,  the two terms on the right hand side of \eqref{eq:integral-2} are equal, and we need estimate only one of them. 

We recall from the proof of Proposition \ref{prop:A1} the random variable $\Gamma_{j+1}$, which corresponds to the probability measure on $\R_+$ given by $c_j x^j d\mu(x)$. As such, we may write
\[\rho(R) : =  \sum_{j=0}^{n-1} \l( \int_0^{ R^2} c_j  x^j d\mu(x) \r) \l( \int_{a^2R^2}^{\infty} c_{j} y^{j} d\mu(y) \r) = \sum_{j=0}^n \P\l[\Gamma_j \le R^2\r] \P\l[ \Gamma_j \ge a^2R^2 \r]. \] It suffices, therefore, to show that $\rho(R) \to 0$ uniformly in $n, \varphi$ as $R \to \infty$.


We consider the $\Gamma_j$-s in two groups: $J_1:=\{ j| \mu_j \le 2 R^2\}$ and $J_2:=\{j| \mu_j > 2 R^2\}$. Let $j_*:=\max\{j|j\in J_1\}$. Clearly,
\begin{align*}  \rho(R) = & \sum_{j=0}^n \P\l[\Gamma_j \le R^2\r] \P\l[ \Gamma_j \ge a^2R^2 \r] \\ \le & \l( \sum_{j \in J_1} \P\l[ \Gamma_j \ge a^2R^2 \r] \r) + \l( \sum_{j \in J_2} \P\l[ \Gamma_j \le R^2 \r] \r). \end{align*}
Observe that, for any $j$ we have
 \[ \E[\Gamma_j]=\mu_j; \quad  \mathrm{Var}\l( \Gamma_j/\mu_j \r)= \l( \frac{\mu_{j+1}}{\mu_j}-1 \r)\] 
 and 
 \begin{equation} \label{eq:nu_j}
 \E\l[\l(\frac{\Gamma_j}{\mu_j}-1\r)^4\r] =\l( \frac{\mu_{j+3}\mu_{j+2}\mu_{j+1}}{\mu_j^3} - 4 \frac{\mu_{j+2} \mu_{j+1}}{\mu_j^2} + 6 \frac{\mu_{j+1}}{\mu_j} -3 \r) = \nu_j .
 \end{equation} 

By  Chebyshev-type  inequality, we have 
\begin{align*} 
& \P\l[ \Gamma_j \ge a^2R^2  \r] \le \frac{\E\l( \Gamma_j/\mu_j -1 \r)^4}{(a^2R^2/\mu_j -1)^4}  \\ 
& \Rightarrow \sum_{j \in J_1} \P\l[ \Gamma_j \ge a^2R^2  \r] \le \sum_{j \in J_1}  \frac{\E\l( \Gamma_j/\mu_j -1 \r)^4}{(a^2R^2/\mu_j -1)^4}. 
\end{align*} 

We recall that $a \ge 2$, so that $\frac{a^2}{2} \ge 2$. 
 For $j \in J_1$ (recall that $J_1:=\{ j| \mu_j \le 2 R^2\}$), we have $\frac{a^2R^2}{\mu_j} \ge \frac{a^2}{2} \ge 2 ~\forall j \in J_1$. Since $|x-1|\ge x/2 ~\forall x \ge 2$, we deduce that  
    $|\frac{a^2R^2}{\mu_j}-1|\ge  \frac{a^2R^2}{2 \mu_j} \forall j \in J_1 $.  
    
Combining with \eqref{eq:nu_j}, We may therefore conclude that for each $j \in J_1$ we have    \[  \frac{\E\l( \Gamma_j/\mu_j -1 \r)^4}{(a^2R^2/\mu_j -1)^4} \le c \cdot  \mu_j^4\nu_j  \big/ (a^2R^2)^4 \] for some constant $c>0$. 
    
    Hence we have, for some constant $C>0$,
\begin{align*} & \sum_{j \in J_1} \P\l[ \Gamma_j \ge a^2R^2  \r] \\ & \le C \cdot \frac{1}{(a^2R^2)^4}\sum_{j \in J_1} \mu_j^4\nu_j \\ & =  C \cdot \frac{\mu_{j_*}^4}{(a^2R^2)^4}\frac{1}{\mu_{j_*}^4} \l\{ \sum_{j \le j_*} \mu_j^4\nu_j\r\} . \end{align*}
Notice that, by definition of $J_1$, we have $\frac{\mu_{j_*}^4}{(a^2R^2)^4} \le 16/a^8$.  By our assumption on the moment sequence, $\l\{ \sum_{j \le j_*} \mu_j^4\nu_j\r\}=o(\mu_{j_*}^4)$ as $j_* \to \infty$, i.e. as $R \to \infty$ (since $\mu_j \to \infty$ as $j \to \infty$). 


On the other hand,  we note for $j \in J_2$ that we have
 \[ \P\l[\Gamma_j \le R^2\r] \le \frac{\E\l( \Gamma_j/\mu_j -1 \r)^4}{(R^2/\mu_j -1)^2} . \]
 Since $\mu_j \ge 2R^2$, we have $|\frac{R^2}{\mu_j}-1|\ge \frac{1}{2}$. As a result, we have \[ \sum_{j \in J_2} \P\l[ \Gamma_j \le R^2  \r] \le C \sum_{j \in J_2} \E\l( \Gamma_j/\mu_j -1 \r)^4  = C \sum_{j > j_*} \nu_j,\] which $\to 0$ as $j_* \to \infty$ (equivalently, as $R \to \infty$) by our assumption on the moment sequence $\{\nu_j\}_j$.

Thus, $\sum_{j=0}^n \P\l[\Gamma_j \le R^2\r] \P\l[ \Gamma_j \ge a^2R^2 \r] \to 0$ as $R \to \infty$, uniformly in $n$ and $\varphi$. This completes the proof of the proposition. 

\end{proof}


We are now ready to complete the proof of Theorem \ref{dpprt}.

\begin{proof}[Proof of Theorem \ref{dpprt}]
 Let $D$ be a bounded open set in $\C$. We will first establish the rigidity of the number of points in $D$.

 To do so, for any $\eps>0$, we will construct a function $\Psi_{\eps}$ with the following properties:
 \begin{itemize}
  \item $\Psi_{\eps} \equiv 1$ on $D$
  \item $\| \Psi_{\eps} \|_\infty \le 1$
  \item $\kappa(\Psi_\eps) < \sqrt{\eps}$.
\end{itemize}
Then, considering the dilations $(\Psi^{\eps})_R$, we get via Proposition \ref{est} that  for large enough $R$ we have
\begin{itemize}
  \item $(\Psi_{\eps})_R \equiv 1$ on $D$
  \item $\| (\Psi_{\eps})_R \|_\infty \le 1$
  \item $\Var\l( \int (\Psi^{\eps})_R d[\Pi] \r) < C \eps $.
\end{itemize}
Appealing to Theorem 6.1 in \cite{GP}, we deduce that the number of points in $D$ is rigid.

Let $r_0$ denote the radius of the disk $D$. We construct the function $\Psi_{\eps}$ as follows:
\[\Psi_{\eps}(r) = \begin{cases}
                           1 & \text{ for $0\le r \le r_0$}, \\
                           -\eps \log r + \eps \log r_0 +1  & \text{ for $r_0 \le r \le r_0\exp(1/\eps)$}, \\
                           0 & \text{ for $r \ge r_0\exp(1/\eps)$ }.
                          \end{cases}
  \]
It is a simple computation to check that $\Psi_{\eps}$ defined above has the desired properties.

\end{proof}

%
%
%


\section{Perturbed lattice in one dimension}
\label{lattice}
Consider $\Pi=\{X_{k} \ : \ k\in \Z\}$ where $X_{k}\sim N(k,\sigma_{k}^{2})$ are independent random variables. We want to consider the special case of $\sigma_{k}^{2}=k^{2\beta}$ and understand the issue of rigidity and tolerance, depending on the parameter $\beta$.

\underline{${\beta>1/2: }$}

For simplicity of notation, we work with perturbed $\N$, i.e., $\{X_{k}\text{ such that } k\ge 1\}$. The proof for the case of $\Z$ (instead of $\N$) is similar. Let $\mu=\mathcal L(X_{1},X_{2},\ldots)$ and $\nu=\mathcal L(X_{2},X_{3},\ldots)$, where by the notation $\mathcal{L}(\underline{Y})$ we denote the law of the random sequence $\underline{Y}$. We claim that $\mu$ and $\nu$ are mutually absolutely continuous and hence, the process $\Pi$ is insertion tolerant (that is, given any bounded domain $D$ and the point configuration  $ \Pi$ we can introduce a new point in $D$ and still remain in the support of $\Pi$). In particular, this implies that there is no rigidity of numbers in $\Pi$. To see this, note that the function mapping the sequence $(X_1,X_2,\cdots,\infty)$ to the point set $\{X_1,X_2\cdots,\infty\}$ is measurable, hence $\mu \equiv \nu$ implies $\Pi \equiv \widetilde{\Pi}$ (where $\widetilde{\Pi}$ is the point process corresponding to $(X_2,\cdots,\infty)$), hence these have the same support. But by a trivial coupling, $\Pi$ contains one more point than the corresponding realization of $\widetilde{X}$, which (being a Gaussian perturbation of $1$) could be inside any given interval in $\R$ with positive probability . Since the distributions of $\Pi$ and $\widetilde{X}$ have the same support, we conclude that the point process $\Pi$ is insertion tolerant in the sense described above.

To compare $\mu$ and $\nu$, we will invoke Kakutani's Dichotomy Theorem. Fix $k$ and let $\mu=k$, $\mu'=k+1$, $\sigma=k^{\beta}$ and $\sigma'=(k+1)^{\beta}$. Then, the Radon-Nikodym derivative of $X_{k}$ to $X_{k+1}$ is
\begin{align*}
f_{k}(x) &= \frac{\sigma'}{\sigma}\exp\l\{-\half\l[\frac{(x-\mu)^{2}}{\sigma^{2}}-\frac{(x-\mu')^{2}}{\sigma'^{2}} \r]\r\}.
\end{align*}
Therefore,
\begin{align*}
\int \sqrt{f_{k}}d\nu &= \sqrt{\frac{\sigma'}{\sigma}}\frac{1}{\sigma'\sqrt{2\pi}}\int \exp\l\{-\frac{1}{4}\l[\frac{(x-\mu)^{2}}{\sigma^{2}}-\frac{(x-\mu')^{2}}{\sigma'^{2}} \r]\r\} \exp\l\{-\half \frac{(x-\mu')^{2}}{\sigma'^{2}}\r\} dx \\
&= \frac{1}{\sqrt{\sigma\sigma'} \sqrt{2\pi}} \int \exp\l\{-\frac{1}{4}\l[\frac{(x-\mu)^{2}}{\sigma^{2}}+\frac{(x-\mu')^{2}}{\sigma'^{2}} \r]\r\} dx \\
&= \frac{1}{\sqrt{\sigma\sigma'} \sqrt{2\pi}} \int \exp\l\{-\frac{1}{4\sigma^{2}\sigma'^{2}}(ax^{2}+bx+c)\r\}dx
\end{align*}
where $a=\sigma^{2}+\sigma'^{2}$, $b=-2\mu\sigma'^{2}-2\mu'\sigma^{2}$ and $c=\mu^{2}\sigma'^{2}+\mu'^{2}\sigma^{2}$. Note that $\int e^{-\half(Ax^{2}+Bx+C)}dx = \sqrt{\frac{2\pi}{A}}e^{(B^{2}-4AC)/2A}$. In our case, $A=a/2\sigma^{2}\sigma'^{2}$, $B=b/2\sigma^{2}\sigma'^{2}$ and $C=c/2\sigma^{2}\sigma'^{2}$ and hence
\begin{align*}
\int \sqrt{f_{k}}d\nu &= \frac{1}{\sqrt{\sigma\sigma'}\sqrt{2\pi}}\frac{\sqrt{2\pi}\sqrt{2\sigma^{2}\sigma'^{2}}}{\sqrt{a}}\exp\l\{-\frac{(\mu-\mu')^{2}}{\sigma^{2}+\sigma'^{2}}\r\} \\
&= \sqrt{\frac{2\sigma\sigma'}{\sigma^{2}+\sigma'^{2}}}\exp\l\{-\frac{(\mu-\mu')^{2}}{\sigma^{2}+\sigma'^{2}}\r\} \\
&= \sqrt{1-\frac{(\sigma-\sigma')^{2}}{\sigma^{2}+\sigma'^{2}}}\exp\l\{-\frac{(\mu-\mu')^{2}}{\sigma^{2}+\sigma'^{2}}\r\}.
\end{align*}
Plugging the values of $\mu=k$, $\sigma=k^{\beta}$ etc., we get
\begin{align*}
\int \sqrt{f_{k}}d\nu &= \sqrt{1-O\l(\frac{1}{k^{2}}\r)}\exp\l\{-O\l(\frac{1}{k^{2\beta}}\r)\r\}
\end{align*}
from which it immediately follows that $\prod_{k}(\int \sqrt{f_{k}}d\nu) >0$. By Kakutani's Theorem, $\mu \equiv \nu$.



\underline{${\beta \le 1/2: }$}

For a compactly supported test function $h$, let $h_{L}(x)=h(x/L)$. Let $\sigma^{2}(h_{L})=\Var(\sum_{x\in \Pi}h_{L}(x))=\sum_{k\in \Z}\Var(h_{L}(X_{k}))$. We shall consider $h $ to be a  non-negative Lipschitz function compactly supported on $[-1,1]$ such that $h \equiv 1$ in a neighbourhood of the origin, and $\|h\|_\infty =1$. Then, $|h_{L}(x)-h_{L}(y)|\le \kappa(h) |x-y|/L$, where $\kappa(h)$ is the Lipschitz constant of $h$. Hence we have
$$
\Var(h_{L}(X_{k}))=\half \E[(h_{L}(X_{k})-h_{L}(Y_{k}))^{2}] \le \frac{\kappa(h)^2}{2L^{2}}\E[(X_{k}-Y_{k})^{2}]=\kappa(h)^2\frac{\sigma_{k}^{2}}{2L^{2}}.
$$
We use this bound for $|k|\le 10L$ and get
\begin{align}\label{eq:boundforupto10L}
\sum_{|k|\le 10L}\Var(h_{L}(X_{k})) \le \frac{\kappa(h)^2}{L^{2}}\sum_{k=1}^{10L}k^{2\beta}  \lesssim \kappa(h)^2 L^{-1+2\beta}.
\end{align}

 For $|k|>10L$, we note that $h_{L}(X_{k})-h_{L}(Y_{k})=0$ unless one of $X_{k}$ or $Y_{k}$ falls in the interval $[-L,L]$. In any case, the difference is bounded by $1$, because we have $0 \le h(x) \le 1$ for any $x$. Hence (for simplicity of notation, let $k$ be positive),
$$
\Var(h_{L}(X_{k}))=\half \E[(h_{L}(X_{k})-h_{L}(Y_{k}))^{2}] \le \P\{|X_{k}-k|\ge k-L\}\le e^{-(k-L)^{2}/2\sigma_{k}^{2}}.
$$
For $k>10L$, we see that $(k-L)^{2}\ge k^{2}/4$ while $\sigma_{k}^{2}=k^{2\beta}\le k$ and hence
\begin{align}\label{eq:boundbeyond10L}
\sum_{|k|>10L}\Var(h_{L}(X_{k})) \le e^{-L/10}.
\end{align}
From \eqref{eq:boundforupto10L} and \eqref{eq:boundbeyond10L} we see that $\sigma^{2}(h_{L})\le C \kappa(h)^2 L^{-1+2\beta}$ (for large enough $L$, depending on $h$).

Suppose we want to prove that the number of points of $\Pi$ in a bounded interval $D \subset \R$ is rigid. For $\eps>0$, we define $h_\eps$ as $\Psi_\eps$ in the proof of Theorem \ref{dpprt} in Section \ref{dpprtproof}. It is easy to see that, for large enough $L$ (depending on $\eps$), $\sigma^{2}(h_{L}) \le c\eps^2$. We can again invoke Theorem 6.2 from \cite{GP} to show that we have the desired rigidity of the number of points in $D$.

Now fix any interval $[-A,A]$ and condition on $\Pi\cap[-A,A]^{c}$. Find $L$ large enough so that $\sigma^{2}(h_{L})<\eps^{3}$. Coupled with Theorem 6.2 in \cite{GP}, this suffices to establish that the number of points in $A$ is rigid.

%
%

\underline{Strong tolerance for $\beta \le 1/2$:}

By looking at the random sequence $(X_1,X_2,\cdots,\infty)$, we notice that for each realization of the random sequence, the points that occur inside any bounded interval $A$ have a joint density (coming from the Gaussians). Hence could be in any other possible locations in $A$ (w.r.t. the Lebesgue measure) with positive probability density. This completes the proof for $\beta  \ge 1/2$.

\section{Appendix: Paradigms for Rigidity and Tolerance}\label{sec:thiswillberewritten} 
In this section,  we establish technical results pertaining to various notions related to rigidity and tolerance. 

We begin by recalling Theorem 6.1 from \cite{GP}, which has been demonstrated to be useful for our purposes in this paper. In fact, in what follows, we state and prove a slightly more general version of this theorem.
\begin{theorem}[Theorem 6.1 in \cite{GP}] \label{thm:GP}
 \label{rig}
Let $\Pi$ be a point process on second countable Hausdorff locally compact topological space $\Xi$, and let $D$ be a measurable  set in $\Xi$. Let $\mathcal{S}_D$ denote the space of locally finite point configurations on $D$.  Let $\varphi$ be a compactly supported measurable function on $\Xi$. Suppose that for any $1>\eps>0$, we have a compactly supported measurable function $\Phi^{\eps}$ such that $\Phi^{\eps}=\varphi$ on $D$, and $\mathrm{Var}\l( \sum_{z \in \Pi} \Phi^{\eps}(z)  \r) < \eps$. Then the function $T:\mathcal{S}_D \to \C$ given by $T(\Upsilon)=\sum_{z \in \Upsilon \cap D} \varphi(z)$ is rigid.
\end{theorem}

\begin{proof}
In the rest of this proof, for any locally finite point configuration $\Upsilon$ and a measurable function $\psi$ on $\Xi$, we will denote by $\int \psi \mathrm{d}[\Upsilon]$ the sum $\sum_{z \in \Upsilon} \psi(z)$.

Consider the sequence of functions $\Phi^{2^{-n}}, n \ge 1$. Note that $\E \l[ \int_{\Xi} \Phi^{2^{-n}} \mathrm{d}[\Pi] \r] = \int_{\Xi} \Phi^{2^{-n}} \mathrm{d} \rho_1 $ where $\rho_1$ is the first intensity measure of $\Pi$.  It follows from Chebyshev's inequality that \[ \P \l(  \l| \int_{\Xi} \Phi^{2^{-n}} \mathrm{d}[\Pi] - \E \l[ \int_{\Xi} \Phi^{2^{-n}} \mathrm{d}[\Pi] \r] \r| >  2^{-n/4} \r) \le 2^{-n/2}.\]  The Borel Cantelli lemma implies that with probability 1, as $n \to \infty$ we have \[\label{rig-1} \l| \int_{\Xi} \Phi^{2^{-n}} \mathrm{d}[\Pi] - \E \l[ \int_{\Xi} \Phi^{2^{-n}} \mathrm{d}[\Pi] \r] \r| \to 0. \]
But  \[ \int_{\Xi} \Phi^{2^{-n}} \mathrm{d}[\Pi] =  \int_{D} \Phi^{2^{-n}} \mathrm{d}[\Pi] +  \int_{D^c} \Phi^{2^{-n}} \mathrm{d}[\Pi]. \]
Thus we have, as $n \to \infty$  
\begin{equation} 
\label{rig-2}
\l| \int_{D} \Phi^{2^{-n}} \mathrm{d}[\Pi] +  \int_{D^c} \Phi^{2^{-n}} \mathrm{d}[\Pi] - \int_{\Xi} \Phi^{2^{-n}} \mathrm{d} \rho_1 \r| \to 0. 
\end{equation}
If we know $\Pi_{D^c}$, we can compute $ \int_{D^c} \Phi^{2^{-n}} \mathrm{d}[\Pi]$ exactly, also $\rho_1$ is known explicitly as a functional of the point process $\Pi$. Hence, from the limit in \eqref{rig-2}, a.s. we can obtain $\int_{D} \Phi^{2^{-n}} \mathrm{d}[\Pi] = \int_{D} \varphi \mathrm{d}[\Pi]$   as the limit
\[ \lim_{n \to \infty} \l( \int_{\Xi}  \Phi^{2^{-n}} \mathrm{d} \rho_1 - \int_{D^c} \Phi^{2^{-n}} \mathrm{d}[\Pi] \r) . \]
\end{proof}


We move on with a lemma that pertains to the nature of the domains for which rigidity of numbers fails to hold (c.f. \ref{rem:inclusion}).

\begin{lemma} \label{lem:inclusion}
Let $\Pi$ be a point process on a space $E$ such that the point count in a domain $U$, i.e. $N_U(\Pi)$, is not rigid. Then, for any domain $V \subset E$ such that $U \subseteq V$, the point count in {\color{red} $V$}, i.e. $N_V(\Pi)$, is also not rigid.
\end{lemma}
\begin{proof} Suppose $N_V(\Pi)$ is rigid. Write  $N_U(\Pi)=N_V(\Pi)-N_{V\cap U^c}(\Pi)$. The second term is clearly in $\sigma(\Pi_{U^c})$, while the rigidity of $N_V(\Pi)$ implies that the first is in $\sigma(\Pi_{V^c})$ which is a subset of $\sigma(\Pi_{U^c})$. Therefore, $N_U(\Pi)\in \sigma(\Pi_U^c)$, showing that $N_U(\Pi)$ must also be rigid.
\end{proof}

We now discuss the relationship between tolerance and strong tolerance. For that we state an equivalent condition for tolerance.
\begin{lemma}\label{lem:3versionsoftolerance} Consider the following statements.
\begin{enumerate}[(A)]
\item\label{eq:tlr} Let $(\Pi,D)$ be tolerant subject to $\{f_{1},\ldots ,f_{m}\}$.
\item\label{eq:mdrttlr} With positive probability, the conditional distribution of $\Pi_{D}$ given $\Pi_{D^{c}}$ is mutually absolutely continuous to the conditional distribution of $\Pi_{D}$ given $\sigma\{f_{1}(\Pi_{D}),\ldots ,f_{m}(\Pi_{D})\}$.
\item\label{eq:strngtlr} With probability one, the conditional distribution of $\Pi_{D}$ given $\Pi_{D^{c}}$ is mutually absolutely continuous to the conditional distribution of $\Pi_{D}$ given $\sigma\{f_{1}(\Pi_{D}),\ldots ,f_{m}(\Pi_{D})\}$.
\end{enumerate}
Then $\eqref{eq:strngtlr}\implies \eqref{eq:mdrttlr} \implies \eqref{eq:tlr}$.
\end{lemma}
As \eqref{eq:strngtlr} is the same as strong tolerance and \eqref{eq:tlr} is the same as tolerance, in particular it follows that strong tolerance implies tolerance. The intermediate statement \eqref{eq:mdrttlr} is put in to show that there are other possibilities. Let us first give an example.
\begin{example} Let $\Pi_{1}$ be the infinite Ginibre ensemble. Let $\Pi_{2}$ be a Poisson process with  unit intensity on the plane. Let $\xi$ be a $\mbox{Ber}(1/2)$ random variable. Assume that $\Pi_{1},\Pi_{2},\xi$ are independent. Set $\Pi=\Pi_{1}$ if $\xi=0$ and $\Pi=\Pi_{2}$ is $\xi=1$. Then
$\Pi$ is a stationary point process.

If $D$ is a bounded set, then from $\Pi_{D^{c}}$ we can deduce $\xi$ and then, we may or may not be able to deduce the number of points in $D$ depending on the value of $\xi$. Hence $\Rig_{D}$ is trivial. Further, if $\xi=0$, then the conditional distribution of $\Pi_{D}$ given $\Pi_{D^{c}}$ is Poisson process on $D$. If $\xi=1$, the conditional distribution of $\Pi_{D}$ is supported on collection of $N$ points where $N$ is a function of $\Pi_{D^{c}}$.  Thus, $\Pi$ satisfies \eqref{eq:tlr} and \eqref{eq:mdrttlr} but not \eqref{eq:strngtlr}.
\end{example}
Playing the same game with Ginibre and GAF zeroes, we can get a stationary point process which is rigid for numbers, which is tolerant subject to numbers, but which is not strongly tolerant subject to numbers.

One might wonder if the point process in the above example is somehow unnatural. For instance, it is not ergodic. Here is an example to show that such a difference can occur even under ergodicity, mixing and tail triviality.
\begin{example} Let $\xi_{n}$, $n\in \Z$, be i.i.d. with $\P\{\xi_{n}=0\}=\P\{\xi_{n}=1\}=\frac12$. Define
\[
X_{n}=\begin{cases}
1 & \mbox{ if }\xi_{n}=1\mb{ or }\xi_{n-1}=\xi_{n+1}=1, \\
0 &\mbox{ otherwise}.
\end{cases}
\]
Then $(X_{n})_{n\in \Z}$ is stationary. Since $(X_{k})_{k\le n}$ and $(X_{k})_{k\ge n+3}$ are independent (a property known as $2$-dependent), $X$ is tail-trivial, ergodic and mixing in every sense.

Condition on $(X_{k})_{k\not= 0}$. If $X_{1}=X_{-1}=1$, then we can easily deduce that $X_{0}=1$. However, for any of the other three possible values for $(X_{1},X_{-1})$, the conditional distribution of $X_{0}$ gives positive probability to both $0$ and $1$. It should also be noted that $(X_{1},X_{-1})$ takes all four values with positive probability.

Regard $X$ as a point process on $\Z$ and let $D=\{0\}$. Then $\Rig_{D}$ is trivial, \eqref{eq:mdrttlr} holds, but not \eqref{eq:strngtlr}.
\end{example}
It is also possible to satisfy \eqref{eq:tlr} without satisfying \eqref{eq:mdrttlr}.


\begin{example} Let $\Omega=\{1,2\}\times \{1,2,3\}$ with uniform probability distribution. Let $X(i,j)=j$ and
$$
Y(1,2)=Y(1,3)=1, \;\; Y(2,1)=Y(2,3)=2, \;\; Y(1,1)=Y(2,2)=3.
$$
Then, $\sigma\{X\}\cap \sigma\{Y\}$ is trivial. However, for any $\omega\in \Omega$, the conditional distribution of $X$ given $Y$ is supported on exactly two of the points $1,2,3$.

To make an example involving point processes, let $(i,j)$ be sampled uniformly from $\Omega$. Given $(i,j)$, let $\Pi$ be a Poisson process on $\R$ with intensity $X(i,j)$ on $(-1,1)$ and intensity $Y(i,j)$ on $(-1,1)^{c}$. Then the process $\Pi$ satisfies \eqref{eq:tlr} but not \eqref{eq:mdrttlr}.
\end{example}

\begin{proof}[Proof of Lemma~\ref{lem:3versionsoftolerance}] We only need to prove that \eqref{eq:mdrttlr} implies \eqref{eq:tlr}.

If \eqref{eq:tlr} were not true, we could find $f:\S_{D}\mapsto \R$ such that
\begin{inparaenum}[(a)]
\item $f$ is rigid for $\Pi$,
\item there is no random variable measurable with respect to $\sigma\{f_{1}(\Pi_{D}),\ldots ,f_{m}(\Pi_{D})\}$ that is equal to $f(\Pi_{D})$ $a.s.$
\end{inparaenum}
In particular, $f(\Pi_{D})$ is not a constant and in fact, with positive probability, the conditional distribution of $f(\Pi_{D})$ given $\sigma\{f_{1}(\Pi_{D}),\ldots ,f_{m}(\Pi_{D})\}$ is not degenerate. Hence, with positive probability, the  conditional distribution of $\Pi_{D}$ given $\sigma\{f_{1}(\Pi_{D}),\ldots ,f_{m}(\Pi_{D})\}$ gives probability less than one to any level set of $f$. But the rigidity of $f$ implies that, almost surely, the conditional distribution of $\Pi_{D}$ given $\Pi_{D^{c}}$ is supported inside a level set of $f$.  This contradicts \eqref{eq:mdrttlr}.
\end{proof}

\section{Acknowledgements} 

S.G. was supported in part by the MOE grants R-146-000-250-133, R-146-000-312-114 and T2EP20121-0033. The authors would like to thank the anonymous referees for their careful reading of the manuscripts, and their insightful comments and suggestions that helped greatly in improving the presentation of the paper.

\begin{figure}[t]\label{fig:alphagafzerosets}
\centering
\includegraphics[scale=1]{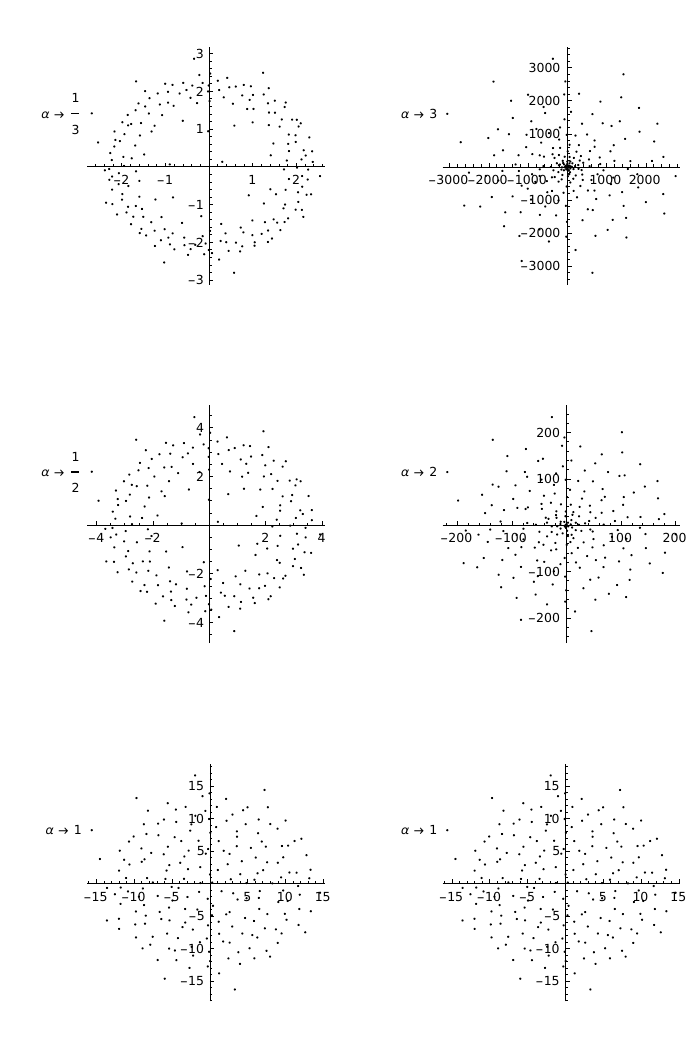}
\caption{Simulations of zeros of $\alpha$-GAFs for a few values of $\alpha$. The power series is truncated at $200$ terms.}
\end{figure}


\begin{thebibliography}{9}

\bibitem[And]{And}
Anderson, J.W., \emph{Hyperbolic geometry}, Springer Science \& Business Media, 2006.

\bibitem[AGL]{AGL}
Adhikari, K.,  Ghosh, S. and Lebowitz, J.L., \emph{Fluctuation and Entropy in Spectrally Constrained random fields}, Communications in Mathematical Physics (2021): 1-32.

\bibitem[BBK]{bbkbook}
Baccelli, Francois and Blaszczyszyn, Bartlomiej and Karray, Mohamed, \emph{Random Measures, Point Processes, and Stochastic Geometry}, Inria, \href{https://hal.inria.fr/hal-02460214/file/PointProcesses45.pdf}{Available online}, 2020.

\bibitem[BQ]{buf-1}
Bufetov, A. I. and  Qiu, Y., \emph{Determinantal point processes associated with Hilbert spaces of holomorphic functions}, Communications in Mathematical Physics 351, no. 1 (2017): 1-44.

\bibitem[BDQ]{buf-2}
Bufetov, A.I., Dabrowski, Y.,  and Qiu, Y., \emph{Linear rigidity of stationary stochastic processes}, Ergodic Theory and Dynamical Systems 38, no. 7 (2018): 2493-2507.

\bibitem[Buf-3]{buf-3}
Bufetov, A.I., \emph{Rigidity of Determinantal Point Processes with the Airy, the Bessel and the Gamma Kernel}, Bulletin of Mathematical Sciences 6, no. 1 (2016): 163-172.

\bibitem[BQS]{BQS}
Bufetov, A. I.,  Qiu,  Y., and Shamov, A., \emph{Kernels of conditional determinantal measures and the Lyons–Peres completeness conjecture}, Journal of the European Mathematical Society 23, no. 5 (2021): 1477-1519.

\bibitem[DaV]{DV}
Daley, D. J. and Vere Jones, D. \emph{An Introduction to the Theory of Point Processes (Vols. I \& II)}, Springer, 1997.

\bibitem[DeHLM]{DHLM}
Dereudre, D., Hardy, A., Leble, T. and Maida, M. \emph{DLR equations and rigidity for the Sine-beta process}, Communications on Pure and Applied Mathematics 74, no. 1 (2021): 172-222..

\bibitem[Du]{Durrett}
Durrett, R., \emph{Probability: theory and examples}, Vol. 49. Cambridge university press, 2019.

\bibitem[DM]{DymMackean}
Dym, H. and  McKean, H. P., \emph{Gaussian processes, function theory, and the inverse spectral problem}, Probability and Mathematical Statistics, Vol. 31. Academic Press.

\bibitem[Fe]{Fe}
Feldheim, N., \emph{Zeroes of Gaussian analytic functions with translation-invariant distribution}, Israel Journal of Mathematics 195, no. 1 (2013): 317-345.

\bibitem[Ge]{Ge}
H.-O. Georgii, \emph{Canonical and grand canonical Gibbs states for continuum systems}, Communications in Mathematical Physics 48, no. 1 (1976): 31-51.

\bibitem[G-I]{G-1}
Ghosh, S., \emph{Determinantal processes and completeness of random exponentials: the critical case}, Probability Theory and Related Fields 163, no. 3-4 (2015): 643-665.


\bibitem[G-II]{G-2}
Ghosh, S., \emph{Palm measures and rigidity phenomena in point processes}, Electronic Communications in Probability 21 (2016).

\bibitem[Gh-III]{G-3}
Ghosh, S., \emph{Quantitative estimates for rigidity and tolerance in generalised Gaussian analytic function zeros and related processes}, in preparation.

\bibitem[GhKr]{GhK}
Ghosh, S., and Krishnapur, M., \emph{Rigidity hierarchy in random point fields: random polynomials and determinantal processes}, arXiv preprint arXiv:1510.08814 (2015).

\bibitem[GhL]{GhL}
Ghosh, S., and Lebowitz, J.L., \emph{Generalized stealthy hyperuniform processes: Maximal rigidity and the bounded holes conjecture}, Communications in Mathematical Physics 363, no. 1 (2018): 97-110.


\bibitem[GP]{GP}
Ghosh, S.  and Peres, Y.,  \emph{Rigidity and Tolerance in point processes: Gaussian zeroes and Ginibre eigenvalues}, http://arxiv.org/abs/1211.2381, Duke Mathematical Journal 166, no. 10 (2017): 1789-1858.

\bibitem[GKP]{GKP}
Ghosh,  S., Krishnapur, M. and Peres, Y., \emph{Continuum Percolation for Gaussian zeroes and Ginibre eigenvalues}, http://arxiv.org/abs/1211.2514v1.


\bibitem[Gin]{Gin}
Ginibre, J., \emph{Statistical ensembles of complex, quaternion, and real matrices}, Journal of Mathematical Physics, (1965).


\bibitem[Go]{Go}
Goldman, A., The Palm measure and the Voronoi tessellation for the Ginibre process. Ann.
Appl. Probab. 20 (2010), 1, 90–128.

\bibitem[HiShTr]{HST}
Hiraoka, Y.,  Shirai, T.  and Trinh, K.D., \emph{Limit theorems for persistence diagrams}, The Annals of Applied Probability 28, no. 5 (2018): 2740-2780.

\bibitem[HS]{HS}
Holroyd, A. E. and Soo, T., \emph{Insertion and deletion tolerance of point processes}, Electron. J. Probab. 18, (2013), no. 74.

\bibitem[HKPV]{HKPV}
Hough, J. B., Krishnapur, M., Peres, Y. and Vir\'{a}g, B., \emph{Zeros of Gaussian Analytic Functions and Determinantal Point Processes}, Providence, RI, American Mathematical Society, (2009).


\bibitem[Ka]{Ka}
Kallenberg, O.,  \emph{Random Measures}, Akademie-Verlag $\cdot$ Berlin and Academic Press, 1983.

\bibitem[KiNi]{KiNi}
Kiro, A., and Nishry, A., \emph{Rigidity for zero sets of Gaussian entire functions}, Electronic Communications in Probability 24 (2019).

\bibitem[Lob]{Lob}
Lobachevskii, N.I., \emph{Pangeometry}, Vol. 4. European Mathematical Society, 2010.

\bibitem[Ly-I]{Ly}
Lyons, R.,  \emph{Determinantal Probability Measures}, Publ. Math. Inst. Hautes Etudes Sci., 98, 167-212, 2003.

\bibitem[Ly-II]{Lyons-ICM}
Lyons, R., \emph{Determinantal probability: basic properties and conjectures}, Proc. International Congress of Mathematicians 2014, Seoul, Korea, vol. IV, 137--161.


\bibitem[LySt]{LySt}
Lyons, R. and Steif, J.,  \emph{Stationary determinantal processes: Phase multiplicity, Bernoullicity, entropy, and domination}, Duke Math. J., Volume 120, Number 3 (2003), 515--575.

\bibitem[NS-I]{NS}
Nazarov, F. and Sodin, M., \emph{Random Complex Zeroes and Random Nodal Lines}, Proceedings of the International Congress of Mathematicians, Vol. III, 1450--1484, 2010.

\bibitem[NS-II]{NS1}
Nazarov, F. and Sodin, M., \emph{Fluctuations in Random Complex Zeroes: Asymptotic Normality Revisited}, Int. Math. Res. Not., Vol. 2011, No. 24, 5720--5759, 2011.

\bibitem[NSV]{NSV}
Nazarov, F., Sodin, M. and  Volberg, A., \emph{Transportation to random zeroes by the gradient flow. Geometric and Functional Analysis}, Vol 17-3, 887--935, 2007.

\bibitem[Os]{Os}
Osada, H.,  \emph{Interacting Brownian motions in infinite dimensions with logarithmic interaction potentials}, Annals of Probability, Vol.41, 2013.

\bibitem[OsSh]{OsSh}
Osada,  H. and Shirai, T., \emph{Absolute continuity and singularity of Palm measures of the Ginibre point process}, Probability Theory and Related Fields 165, no. 3 (2016): 725-770.

\bibitem[PS]{PS}
Peres, Y.  and Sly, A., \emph{Rigidity and tolerance for perturbed lattices}, arXiv preprint arXiv:1409.4490

\bibitem[Qi]{Qi}
Qiu, Y. \emph{Rigid stationary determinantal processes in non-Archimedean fields}, Bernoulli 25, no. 1 (2019): 75-88.

\bibitem[ST1]{ST1}
Sodin, M. and Tsirelson,  B., \emph{Random complex zeroes. I. Asymptotic normality}, Israel J. Math. {\bf 144} (2004), 125-149.

\bibitem[Sosh]{soshnikov}
Soshnikov, A., \emph{Determinantal random point fields},  Uspekhi Mat. Nauk {\bf 55}, (2000), no. 5(335), 107--160; translation in
Russian Math. Surveys {\bf 55} (2000), no. 5, 923-975

\bibitem[ShTa]{shiraitakahashi}
Shirai, T. and Takahashi, Y., \emph{Random point fields associated with certain Fredholm determinants I: fermion, Poisson and boson point processes}, Journal of Functional Analysis
 {\bf 205}, Issue 2, (2003), 414--463

\bibitem[Ter]{Terras}
Terras, A., \emph{Harmonic analysis on symmetric spaces -- Euclidean space, the sphere, and the Poincare upper half-plane}, Springer Science \& Business Media, 2013.

\bibitem[Tri]{Tripathi}
Tripathi, R.,  \emph{Determinantal Processes and Stochastic Domination}, arXiv preprint arXiv:2009.09141 (2020).


\end{thebibliography}
\end{document}